\documentclass[11pt,british,reqno]{amsart}

\usepackage[T1]{fontenc}

\usepackage{babel,eucal,url,amssymb,enumerate,booktabs,amscd,graphics}

\theoremstyle{plain}
\newtheorem{lemma}{Lemma}[section]
\newtheorem{definition}[lemma]{Definition}
\newtheorem{proposition}[lemma]{Proposition}
\newtheorem{corollary}[lemma]{Corollary}
\newtheorem{theorem}[lemma]{Theorem}
\newtheorem{remark}[lemma]{Remark}

\newcommand{\cyclic}{\mathop{\text{\large$\mathfrak S$}\vrule width 0pt depth 2pt}}

\def\sideremark#1{\ifvmode\leavevmode\fi\vadjust{\vbox to0pt{\vss
 \hbox to 0pt{\hskip\hsize\hskip1em
 \vbox{\hsize2.5cm\tiny\raggedright\pretolerance10000
 \noindent #1\hfill}\hss}\vbox to8pt{\vfil}\vss}}}%

\input xy
\xyoption{all}

\title[Invariant contact metric structures]{All invariant contact metric structures on tangent sphere bundles of compact rank-one symmetric spaces}

\author[J.~C.~Gonz{\'a}lez-D{\'a}vila]{J.~C.~Gonz{\'a}lez-D{\'a}vila}
\address{Departamento de Matem\'aticas, Estad\'istica e Investigaci\'on
Ope\-ra\-tiva, University of La Laguna, 38200 La Laguna, Tenerife, Spain.}
\email{jcgonza@ull.es}

\thanks{Partially supported by grant PID2019-105019GB-C21 funded by MCIN/AEI/ 10.13039/501100011033 and by ERDF `A way of making Europe'}


\begin{document}

\maketitle

\begin{abstract} 
All invariant contact metric structures on tangent sphere bundles of each compact rank-one symmetric space are obtained explicitly, distinguishing for the orthogonal case those that are $K$-contact, Sasakian or $3$-Sasakian. Only the tangent sphere bundle $T_{r}{\mathbb C}{\bf P}^{n}$ $(r>0)$ of complex projective spaces admits $3$-Sasakian metrics and there exists a unique orthogonal Sasakian-Einstein metric on $T_{r}{\mathbb C}{\bf P}^{n}.$ Furthermore, there is a unique invariant contact metric that is Einstein, in fact Sasakian-Einstein, on tangent sphere bundles of spheres and real projective spaces.

Each invariant contact metric, Sasakian, Sasakian-Einstein or $3$-Sasakian structure on the unit tangent sphere of any compact rank-one symmetric space is extended, respectively, to an invariant almost K\"ahler, K\"ahler, K\"ahler Ricci-flat or hyperK\"ahler structure on the punctured tangent bundle.
\end{abstract}

\section{Introduction}

 Compact rank-one Riemannian symmetric spaces can be characterized as those Riemannian symmetric spaces $(G/K,g)$ with strictly positive curvature. They are Euclidean spheres, real, complex and quaternionic projective spaces and the Cayley plane \cite[Ch. 3]{Be1}. The natural $G$-action on the tangent bundle $T(G/K)$ of $G/K$ is of cohomogeneity one with principal orbits the tangent sphere bundles $T_{r}(G/K )= \{u\in T(G/K)\colon g(u,u) = r^{2}\}$ of radius $r>0$ (see \cite{JC}). Hence, fixed a Cartan decomposition ${\mathfrak g} = {\mathfrak k}\oplus {\mathfrak m}$ of $G/K$ and a (one-dimensional) Cartan subspace ${\mathfrak a}$ of ${\mathfrak m},$ 
 the punctured tangent bundle $D(G/K):= T(G/K)\setminus \{\mbox{zero section}\}$ of $T(G/K)$ is identified with $G/H\times {\mathbb R}^{+}$ and $T_{r}(G/K),$ for each $r>0,$ with the quotient $G/H$ via $G$-equivariant diffeomorphisms, where $H$ is the closed subgroup of $K$ given by $H = \{k\in K\colon {\rm Ad}_{k} u=u,\;\mbox{\rm for all}\,u\in {\mathfrak a}\}.$ The quotient expressions $G/K$ as symmetric spaces and the Lie groups $H$ are listed in {\rm Table I}. 
 
 Then the tangent sphere bundle $T_{r}{\mathbb S}^{n}$ of the unit $n$-dimensional sphere ${\mathbb S}^{n}\subset {\mathbb R}^{n+1}$ is equivariantly diffeomorphic to the Stiefel manifold $V_{2}({\mathbb R}^{n+1})$ of all orthonormal $2$-frames in ${\mathbb R}^{n+1}$ and $T_{r}{\mathbb K}{\bf P}^{n},$ where ${\mathbb K}{\bf P}^{n}$ is the $n$-dimensional projective space for ${\mathbb K} = {\mathbb R},$ ${\mathbb C}$ or ${\mathbb H},$ to the projective Stiefel manifold $W_{2}({\mathbb K}^{n+1}): = V_{2}({\mathbb K}^{n+1})/S,$ with $S$ being the subset of unit vectors of ${\mathbb K}$ $(S = {\mathbb S}^{0},$ ${\mathbb S}^{1}$ or ${\mathbb S}^{3}).$   
 
 In \cite{JC1} the author has obtained the set of those $G$-invariant contact metric and Sasakian structures on $T_{r}(G/K) = G/H$ whose characteristic vector field is the {\em standard vector field} $\xi^{S}$ of $T(G/K),$ that is, the unit vector field $\xi^{S}: =-J^{S}N$ defined on $D(G/K),$ where $J^{S}$ is the standard almost complex structure and $N$ is the outward normal unit vector field to the singular foliation $\{T_{r}(G/K)\}_{r>0}$ with respect to the Sasaki metric $g^{S}$ on $T(G/K).$ 
 
   \bigskip
\begin{tabular}{llllll}
\multicolumn{6}{c}{Table I. Compact rank-one symmetric spaces}\\
\hline\noalign{\smallskip}
& $G/K$
& {$\dim$}
& $m_\varepsilon$
& $m_{\varepsilon/{\scriptscriptstyle 2}}$ & $H$ \\
\noalign{\smallskip}\hline\noalign{\smallskip}

$\mathbb{S}^n\,(n\geq 2)$ & $\mathrm{SO}(n{+}1)/\mathrm{SO}(n)$ & $n$ & $n{-}1$ & $0$ & ${\mathrm S}{\mathrm O}(n{-}1)$\\ \smallskip

${\mathbb R}{\mathbf P}^n\, (n\geq 2)$ & 
 $\mathrm{SO}(n{+}1)/{\rm S}({\rm O}(1)\times {\rm O}(n))$ &
$n$ & $n{-}1 $ & $0$ & ${\mathrm S}({\mathrm O}(1)\times {\mathrm O}(n{-}1))$ \\ \smallskip

${\mathbb C}{\mathbf P}^n$\,{$(n \geq 2)$}
&{$\mathrm{SU}(n{+}1)/\mathrm{S}(\mathrm{U}(1){\times}
\mathrm{U}(n))$}& $2n$
& $1$ & $2n{-}2$ & ${\mathrm S}({\mathrm U}(1)\times {\mathrm U}(n{-}1))$ \\ \smallskip
${\mathbb H}{\mathbf P}^n$\,{$(n\geq 1)$}
& {$\mathrm{Sp}(n{+}1)/\mathrm{Sp}(1){\times} \mathrm{Sp}(n)$} & $4n$
& $3$ & $4n{-}4$ & ${\mathrm S}{\mathrm p}(1)\times {\mathrm S}{\mathrm p}(n{-}1)$ \\
\smallskip
${\mathbb C\mathrm a}{\mathbf P}^2$
& {$\mathrm{F_4}/\mathrm{Spin(9)}$} & $16$ & $7$
& 8 & $\mathrm{Spin}(7)$ \\
\hline
\end{tabular}

\vspace{0.5cm}

A first objective of this paper is focused on explicitly obtaining {\em all} the $G$-invariant contact metric structures in $T_{r}(G/K)$ and on determining for the orthogonal case those that are $K$-contact, Sasakian or $3$-Sasakian. This leads us to consider the space of $G$-invariant vector fields in $T_{r}(G/K). $ Such space is one-dimensional except for $G/K = {\mathbb C}{\bf P}^{n}$ whose dimension is three. From this, the characteristic vector field of $G$-invariant contact metric structures on $T_{r}G/K,$ for $G/K\in \{{\mathbb S}^{n},{\mathbb R}{\bf P}^{n}(n\geq 3),{\mathbb H}{\bf P}^{n} (n\geq 1), {\mathbb C}{\rm a}{\bf P}^{2}\},$ is proportional to the standard field and, for $G/K\in \{{\mathbb C}{\bf P}^{n}\,(n\geq 1)\},$ the choice of possible characteristic vector fields produces seven different types of invariant contact metric structures, listed as AI-AIII, BI-BIII and C in Theorem \ref{maingeneralCP1}. Furthermore, those orthogonal structures that are $K$-contact (see Theorem \ref{maingeneral} for their classification) are indeed Sasakian (Theorem \ref{t1}) and, on $T_{r}{\mathbb C}{\bf P}^{n},$ the set of these $G$-invariant Sasakian structures form three families, depending on a positive parameter, of type AI, AII and AIII, respectively.

Boyer, Galicki and Mann in \cite{BGM} obtained the classification of homogeneous 3-Sasakian ma\-ni\-folds (see \cite{GRS} for a new proof using theory of root systems). Comparing the lists of homogeneous $3$-Sasakian manifolds and of quotients $G/H =T_{r}G/K,$ from Table I, only ${\mathrm S}{\mathrm p}(1)\cong {\mathbb S}^{3}$ and ${\mathrm S}{\mathrm U}(n+1)/{\mathrm S}({\mathrm U}(1)\times {\mathrm U}(n-1)),$ $n\geq 2,$ are homogeneous $3$-Sasakian manifolds. Hence $T_{r}G/K$ admits a $G$-invariant $3$-Sasakian structure if and only if $G/K= {\mathbb C}{\bf P}^{n},$ for some $n\geq 1.$ In Theo\-rem \ref{t1}, by a direct procedure, the same result is reached, also obtaining that there is a unique orthogonal $3$-Sasakian metric ${\bf g}$ on $T_{r}{\mathbb C}{\bf P}^{n}.$ Its expression, as well as a triple of charac\-te\-ristic vector fields, are explicitly determined there. Moreover, ${\bf g}$ is the unique orthogonal metric on $T_{r}{\mathbb C}{\bf P}^{n}$ which is Sasakian-Einstein (Theorem \ref{t2}).

On the other hand, Back and Hsiang in \cite{BaHs} proved that $V_{2}({\mathbb R}^{n+1})$ admits invariant Einstein metrics and that any two homogeneous Einstein metrics are homothetic. The same result was obtained by M. Kerr in \cite{Ke}. In this paper, apart from explicitly obtaining such metrics for $T_{r}{\mathbb S}^{n}$ and $T_{r}{\mathbb R}{\bf P}^{n}$ (Theorem \ref{teinstein1}), it is proven that there exists a unique invariant contact metric which is Einstein and that it is in fact Sasakian-Einstein (Theorem \ref{teinsteinS1}).

Another relevant fact that we raise is that, because these classes of contact metric structures can be characterized by using holonomy reduction of its metric cone (see \cite{BoyerGalicki:3Sasakian}, \cite{BoyerGalicki} and Proposition \ref{pcontact}), each $G$-invariant contact metric, Sasakian, Sasakian-Einstein or $3$-Sasakian structure on the unit tangent sphere $T_{1}(G/K),$ can be extended respectively to an $G$-invariant almost K\"ahler, K\"ahler, K\"ahler Ricci-flat or hyperK\"ahler structure on the punctured tangent bundle $D(G/K)$ (Theorem \ref{tmain}). However, we must point out that such structures in $D(G/K)$ do not admit differentiable extensions to all $T(G/K)$ (see Remark \ref{rex}).

 To achieve our objectives, we consider the system $\Sigma$ of restricted roots of $({\mathfrak g},{\mathfrak k},{\mathfrak a}),$ which is either $\{\pm \varepsilon\}$ for spheres and real projective spaces or $\Sigma=\{\pm\varepsilon,\pm\varepsilon/2\}$ for the rest of compact rank-one symmetric spaces, $\varepsilon$ being a linear function on the complexification ${\mathfrak a}^{\mathbb C}$ of ${\mathfrak a},$ such that $\varepsilon({\mathfrak a})\subset \sqrt{-1}{\mathbb R}$ (see \cite{JC} for more details). The multiplicities $m_\varepsilon$, $m_{\varepsilon/{\scriptscriptstyle 2}}$ of the corresponding restricted roots appear in Table I.

\section{Preliminaries}\label{Preliminares}
 
\setcounter{equation}{0}

We briefly recall some basic concepts about almost contact metric structures and their most significant classes. For general theory on almost contact metric structures we refer the reader to \cite{Bl} and for their classification to \cite{ChG}.

An odd-dimensional $C^{\infty}$ manifold $M$ is said to be {\em almost contact} if it admits a $(\varphi,\xi,\eta)$-structure, where $\varphi$ is a tensor field of type $(1,1),$ $\xi$ is a vector field and $\eta$ is a $1$-form, such that
\[
\varphi^{2} = -{\rm I} + \eta\otimes \xi,\quad \eta(\xi) = 1.
\]
Then $\varphi\xi = 0$ and $\eta\circ\varphi =0.$ If moreover $M$ is equipped with a Riemannian metric $g$ such that
\[
g(\varphi U,\varphi V) = g(U,V) - \eta(U)\eta(V),
\]
for all $U,V\in {\mathfrak X}(M),$ where ${\mathfrak X}(M)$ is the Lie algebra of the vector fields on $M,$ $(M,\varphi,\xi,\eta,g)$ is called an {\em almost contact metric manifold}. Two almost contact metric manifolds $(M_{1},\varphi_{1},\xi_{1},$ $\eta_{1},g_{1})$ and $(M_{2},\varphi_{2},\xi_{2},\eta_{2},g_{2})$ is said to be i{\em somorphic} if there exists an isometry $f\colon (M_{1},g_{1})\to (M_{2},g_{2})$ preserving the corresponding structures, i.e. $f_{*}\xi_{1} = \xi_{2}$ and $f_{*}\circ \varphi_{1} = \varphi_{2}\circ f_{*}.$ (Then, $f^{*}\eta_{2} = \eta_{1}.)$

Let $\nabla$ be the Levi-Civita connection of $(M,g).$ Then, we have
\begin{equation}\label{nablaPhi}
(\nabla_{U}\Phi)(V,W) = g(V,(\nabla_{U}\varphi)W),\quad (\nabla_{U}\eta) V = g(V,\nabla_{U}\xi) = (\nabla_{U}\Phi)(\xi,\varphi V),
\end{equation}
where $\Phi$ denotes the {\em fundamental $2$-form}, defined by $\Phi(U,V) = g(U,\varphi V),$ and the exterior de\-ri\-va\-tives of $\eta$ and $\Phi$ are given by
\begin{equation}\label{varphi1}
2d\eta(U,V) = (\nabla_{U}\eta)V - (\nabla_{V}\eta)U,\quad 3d\Phi(U,V,W) = \cyclic_{U,V,W}(\nabla_{U}\Phi)(V,W).
\end{equation}
We say that $(\varphi,\xi,\eta,g)$ is a {\em contact metric} structure if $d\eta= \Phi.$ Then it is determined by the pair $(\xi,g)$ (or $(\eta,g)).$ A Riemannian metric $g$ is said to be {\em contact} if there exists a vector field $\xi$ such that $(\xi,g)$ is a contact metric structure and to each of these vector fields $\xi,$  a {\em characteristic vector field} associated with $g.$ If, in addition, $\xi$ is Killing, the manifold is called {\em $K$-contact} and one gets
\begin{equation}\label{varphiK}
\nabla_{U}\xi = -\varphi U.
\end{equation}

An almost contact structure $(\varphi,\xi,\eta)$ is said to be {\em normal} if the $(1,1)$-tensor field ${\mathbf N}(U,V) = [\varphi,\varphi](U,V) + 2{\rm d}\eta(U,V)\xi$ vanishes for all $U,V\in {\mathfrak X}(M),$ where $[\varphi,\varphi]$ is the {\em Nijenhuis torsion} of $\varphi,$ 
 \begin{equation}\label{Nijenhuis}
 [\varphi,\varphi] (U,V) = \varphi^{2}[U,V] + [\varphi U,\varphi V] - \varphi[\varphi U,V] - \varphi[U,\varphi V].
 \end{equation}
 A contact metric structure $(\xi, g)$ which is normal is called a {\em Sasakian structure}. A useful characterization for Sasakian manifolds is the following: An almost contact metric manifold $(\varphi,\xi,\eta,g)$ is Sasakian if and only if  
  \begin{equation}\label{Sasakian}
 (\nabla_{U}\varphi)V = g(U,V)\xi - \eta(V)U,\quad U,V\in {\mathfrak X}(M).
 \end{equation}
Any Sasakian manifold is $K$-contact but the converse may be not true for dimensions greater than three. A Sasakian manifold $(M,\xi,g)$ is {\em Sasakian-Einstein} if its Riemannian metric $g$ is Einstein. Then on a $(2n+1)$-dimensional Sasakian-Einstein manifold, the Ricci tensor ${\rm Ric}$ satisfies ${\rm Ric}(U, \xi) = 2n\eta(U)$ and the scalar curvature is positive and equals to $2n(2n+1).$  
 
 A Riemannian manifold $(M,g)$ is called {\em $3$-Sasakian} if it admits three vector fields $\{\xi_{1},\xi_{2},\xi_{3}\}$ such that
 \[
 g(\xi_{i},\xi_{j}) = \delta_{ij},\quad [\xi_{i},\xi_{j}] = 2\varepsilon_{ijk}\xi_{k},\quad i,j,k\in \{1,2,3\},
 \]
 and $(M,g,\xi_{i})$ is a Sasakian manifold, for each $i=1,2,3.$ Because $(M,g,\xi)$ is a Sasakian manifold, for any $\xi = \lambda_{1}\xi_{1} + \lambda_{2}\xi_{2} + \lambda_{3}\xi_{3}$ where $(\lambda_{1},\lambda_{2},\lambda_{3})\in {\mathbb S}^{2}(1),$ a $3$-Sasakian manifold has in fact an ${\mathbb S}^{2}(1)$ worth of Sasakian structures.
 
 A more recent characterization for Sasakian, Sasakian-Einstein and $3$-Sasakian manifold by using holonomy reduction of the associated metric cone appears in \cite{BoyerGalicki:3Sasakian}. The reader is also referred to \cite{BoyerGalicki:Sasakian-Einstein}, \cite{BoyerGalicki}, \cite{Sparks} and the monograph \cite{BoyerGalickilibro} for more information about this topic.
 
 Given an $m$-dimensional Riemannian manifold $(M,g),$ the {\em metric cone} $({\mathcal C}(M),\overline{g})$ is the warped product $M\times_{{\rm Id}_{{\mathbb R}^{+}}}{\mathbb R}^{+},$ that is the product manifold ${\mathcal C}(M)=M\times {\mathbb R}^{+}$ equipped with the Riemannian metric $\overline{g} = r^{2}g + dr^{2}.$ Then $(M,g)$ is Sasakian if the holonomy group of $({\mathcal C}(M),\overline{g})$ reduces to a subgroup of ${\mathrm U}(\frac{m+1}{2}).$ In particular, $m = 2n + 1,$ $n \geq 1,$ and $({\mathcal C}(M),\overline{g})$ is K\"ahler. For the proof, it is considered the almost Hermitian structure $(J,\overline{g})$ on ${\mathcal C}(M),$ associated to an almost contact metric structure $(\varphi,\xi,\eta)$ on $(M,g),$ where $J$ is given by
\begin{equation}\label{JJ}
J(U,\lambda \frac{\partial}{\partial r}) = (\varphi U - \frac{\lambda}{r}\xi,r\eta(U)\frac{\partial}{\partial r}),\quad \mbox{\rm for all}\;\lambda\colon {\mathbb R}^{+}\to {\mathbb R}\;\mbox{\rm smooth function.}
\end{equation}

\begin{proposition}\label{pcontact} A Riemannian manifold $(M,g)$ is contact metric if and only if the metric cone $({\mathcal C}(M),\overline{g})$ is almost K\"ahler. Moreover, the symplectic $2$-form $F$ in ${\mathcal C}(M)$ is the exact form $F = d(r^{2}\eta),$ where $\eta$ is the contact form of $M.$
\end{proposition}
\begin{proof} From (\ref{JJ}), the corresponding fundamental $2$-forms $F$ and $\Phi$ of $({\mathcal C}(M),J,\overline{g})$ and of the contact metric manifold $(M,\xi,g)$ satisfy $ F = r^{2}\Phi + 2rdr\wedge \eta = d(r^{2}\eta),$ where tensors of $M$ (resp., ${\mathbb R}^{+})$ and their pull-backs on ${\mathcal C}(M)$ via the projection $\pi_{1}\colon {\mathcal C}(M)\to M$ (resp., $\pi_{2}\colon {\mathcal C}(M)\to {\mathbb R}^{+}),$ are denote with the same letter. Then $({\mathcal C}(M),\overline{g})$ is almost K\"ahler. 

Conversely, let $J$ be an almost complex structure on ${\mathcal C}(M)$ compatible with $\overline{g}$ such that $dF = 0,$ $F$ being the K\"ahler form of $({\mathcal C}(M),J,\overline{g}).$ Identifying $M$ with $M\times \{1\}\subset {\mathcal C}(M),$ we consider the induced almost contact metric structure $(\varphi,\xi,\eta,g)$ on $M$ as hypersurface of ${\mathcal C}(M),$ where
\[
\xi = -J\frac{\partial}{\partial r},\quad \eta = g(\xi,\cdot),\quad J(U,0) = (\varphi U,\eta(U)\frac{\partial}{\partial r}).
\]
Then $J$ on $M\cong M\times\{1\}$ coincides with the almost complex structure defined in (\ref{JJ}). Moreover, for the Levi-Civita connection $\overline{\nabla}$ of $({\mathcal C}(M),\overline{g})$ we have
$$
\begin{array}{lcl}
(\overline{\nabla}_{(0,\frac{\partial}{\partial r})}J)(0,\frac{\partial}{\partial r}) & = & (\overline{\nabla}_{(0,\frac{\partial}{\partial r})}J)(U,0) = (0,0),\\[0.4pc]
(\overline{\nabla}_{(U,0)}J))(0,\frac{\partial}{\partial r}) & =  &-\frac{1}{r}(\nabla_{U}\xi + \varphi U,0),\\[0.4pc]
(\overline{\nabla}_{(U,0)}J)(V,0)  & =  &((\nabla_{U}\varphi)V + \eta(V)U - g(U,V)\xi,r g(\nabla_{U}\xi + \varphi U,V)\frac{\partial}{\partial r}).
\end{array}
$$
Now, using (\ref{nablaPhi}) and (\ref{varphi1}), we get at the points $(x,1)\in M\times\{1\},$ 
$$
\begin{array}{lcl}
(\iota_{(0,\frac{\partial}{\partial r})}dF)((U,0),(V,0)) & = & \frac{1}{3}\cyclic(\overline{\nabla}_{(0,\frac{\partial}{\partial r})})((U,0),(V,0)) = \frac{1}{3}\cyclic\overline{g}((U,0),(\overline{\nabla}_{(0,\frac{\partial}{\partial r})}J)(V,0))\\[0.4pc]
& =  &\frac{1}{3}(\overline{g}((V,0),-(\nabla_{U}\xi + \varphi U,0)) + g(\nabla_{V}\xi + \varphi V,U))\\[0.4pc]
& = & \frac{2}{3}(\Phi(U,V) - d\eta(U,V)).
\end{array}
$$
This proves the result.
\end{proof}

Under the identification $M\cong M\times \{1\}\subset{\mathcal C}(M),$ we have also proven the following.
\begin{corollary}\label{cclasses} Each contact metric structure $(\xi,g)$ on an odd-dimensional manifold $M$ can be extended to a unique almost K\"ahler structure on its metric cone $({\mathcal C}(M),\overline{g}).$
\end{corollary}
As a direct consequence of the relationship between Ricci curvature of $M$ and its metric cone $({\mathcal C}(M),\overline{g}),$ it follows that $g$ is Einstein if and only if $\overline{g}$ is Ricci-flat. Then $({\mathcal C}(M),\overline{g})$ is K\"ahler Ricci-flat (Calabi-Yau) and the restricted holonomy group reduces to a subgroup of ${\rm SU}(n + 1),$ where $\dim M = 2n+1.$ Moreover, an $m$-dimensional Riemannian manifold $(M,g)$ is $3$-Sasakian if and only if the holonomy group of the metric cone on $M$ reduces to a subgroup of ${\rm Sp}(\frac{m+1}{4}).$ In particular, $m = 4n+3,$ $n\geq 1,$ and $({\mathcal C}(M),\overline{g})$ is {\em hyperK\"ahler}. This implies that $3$-Sasakian manifolds are Sasakian-Einstein with Einstein constant $\lambda = 2(2n+1).$ 

 \section{Homogeneous almost contact metric manifolds}

\setcounter{equation}{0}

A connected homogeneous manifold $M$ can be expressed as a quotient manifold $G/K,$ where $G$ is a connected Lie group acting transitively on $M$ and $K$ is the isotropy subgroup of $G$ at some point $o\in M,$ the {\em origin} of $G/K.$ Moreover, if $g$ is a $G$-invariant Riemannian metric on $M = G/K,$ then $(M,g)$ is said to be a {\em homogeneous Riemannian manifold} and, if $(\varphi,\xi,\eta)$ is an $G$-invariant almost contact structure on $M$ compatible with $g,$ $(M,\varphi,\xi,\eta,g)$ is said to be a {\em homogeneous almost contact metric manifold}.

We can assume that $G/K$ is a {\em reductive} homogeneous space, i.e. there exists an ${\rm Ad}(K)$-invariant subspace ${\mathfrak m}$ of the Lie algebra ${\mathfrak g}$ of $G$ such that ${\mathfrak g} = {\mathfrak k}\oplus{\mathfrak m},$ ${\mathfrak k}$ being the Lie algebra of $K.$ Then the differential map $(\pi_{K})_{*e}$ at the identity element $e\in G$ of the projection $\pi_{K}\colon G\to G/K,$ $\pi_{K}(a) = aK,$ gives an isomorphism of ${\mathfrak m}$ onto $T_{o}(G/K)$ and $g$ is determined by an ${\rm Ad}(K)$-invariant inner product $\langle\cdot,\cdot\rangle$ on ${\mathfrak m}.$ 

It is clear that there exists a sufficiently small neighborhood $O_{\mathfrak m}\subset {\mathfrak m}$ of zero in ${\mathfrak m}$ such that $\exp(O_{\mathfrak m})$ is a submanifold of $G$ and the mapping $\pi_{K\mid\exp(O_{\mathfrak m})}$ is a diffeomorphism onto a neighborhood ${\mathcal U}_{o}$ of $o.$ Hence, for each $\mu\in {\mathfrak m},$ a vector field $\mu^{\tau}$ on ${\mathcal U}_{o}$ is defined as
\begin{equation}\label{tautau}
\mu^{\tau}_{(\exp x) K} = (\tau_{\exp x})_{*o}\mu,\quad \mbox{\rm for all $x\in O_{\mathfrak m},$}
\end{equation}
where $\tau_{b},$ for each $b\in G,$ is the translation $\tau_{b}\colon G/K\to G/K,$ $\tau_{b}(aK) = baK.$ Then $[\mu_{1}^{\tau},\mu^{\tau}_{2}]_{o} = [\mu_{1},\mu_{2}]_{\mathfrak m}$ (see \cite{N}), where $[\cdot,\cdot]_{\mathfrak m}$ denotes the ${\mathfrak m}$-component of $[\cdot,\cdot].$ 

There exists a one-to-one correspondence between the set of all invariant affine connections $\nabla$ on $G/K$ and the set of all ${\rm Ad}(K)$-invariant bilinear functions $\alpha:{\mathfrak m}\times{\mathfrak m}\to {\mathfrak m}$ \cite[Theorem 8.1]{N}. The correspondence is given by 
\begin{equation}\label{alpha}
\alpha(\mu_{1},\mu_{2}) = \nabla_{\mu_{1}}\mu_{2}^{\tau},\quad \mu_{1},\mu_{2}\in {\mathfrak m}.
\end{equation}
Then the curvature tensor $R$ of $\nabla$ at the origin $o,$ where $R$ is defined by the sign convention $R(U,V) = \nabla_{[U,V]}-[\nabla_{U},\nabla_{V}]$ for all $U,V\in {\mathfrak X}(M),$ is expressed in terms of $\alpha$ by
\begin{equation}\label{R}
\begin{array}{lcl}
R_{o}(\mu_{1},\mu_{2})\mu_{3} & = &  [[\mu_{1},\mu_{2}]_{\mathfrak k},\mu_{3}] + \alpha([\mu_{1},\mu_{2}]_{\mathfrak m},\mu_{3})\\[0.4pc]
& &  + \alpha(\mu_{2},\alpha(\mu_{1},\mu_{3})) -  \alpha(\mu_{1},\alpha(\mu_{2},\mu_{3})),
\end{array}
\end{equation}
for all $\mu_{1},\mu_{2},\mu_{3}\in {\mathfrak m}.$ Using the Koszul formula in (\ref{alpha}), the Levi Civita connection of $(M,g)$ is determined for $\alpha$ given by
\begin{equation}\label{nabla}
\alpha(\mu_{1},\mu_{2}) = \frac{1}{2}[\mu_{1},\mu_{2}]_{\mathfrak m} + {\mathfrak U}(\mu_{1},\mu_{2}),
\end{equation} 
 where ${\mathfrak U}$ is the symmetric bilinear function on ${\mathfrak m}\times{\mathfrak m}$ such that
\begin{equation}\label{U}
2\langle{\mathfrak U}(\mu_{1},\mu_{2}),\mu_{3}\rangle = \langle[\mu_{3},\mu_{1}]_{\mathfrak m},\mu_{2}\rangle + \langle[\mu_{3},\mu_{2}]_{\mathfrak m},\mu_{1}\rangle.
\end{equation}
From (\ref{alpha}), a $G$-invariant vector field $U$ on $G/K$ is Killing if and only if $\langle\alpha(\mu_{1},\mu),\mu_{2}\rangle + \langle\alpha({\mu_{2}},\mu),\mu_{1}\rangle = 0,$ for all $\mu_{1},\mu_{2}\in {\mathfrak m},$ where $\mu $ is the ${\rm Ad}(K)$-invariant vector $\mu = U_{o}\in {\mathfrak m}.$ Since from (\ref{nabla}) and (\ref{U}) we get $\langle\alpha(\mu_{1},\mu),\mu_{2}\rangle + \langle\alpha({\mu_{2}},\mu),\mu_{1}\rangle = -\langle{\mathfrak U}(\mu_{1},\mu_{2}),\mu\rangle,$
 $U$ is a Killing vector field if and only if 
 \begin{equation}\label{Killing}
 {\mathfrak U}(\mu_{1},\mu_{2})\in \mu^{\bot},\quad \mu_{1},\mu_{2}\in {\mathfrak m}.
\end{equation}

The $G$-invariant connection $\widetilde{\nabla}$ corresponding with $\alpha = 0$ is known as the {\em canonical connection} \cite{KNI}. Its torsion $\widetilde{T}$ and curvature $\widetilde{R}$ are given at the origin by
\begin{equation}\label{TR}
\widetilde{T}_{o}(\mu_{1},\mu_{2})  =  -[\mu_{1},\mu_{2}]_{\mathfrak m},\quad \widetilde{R}_{o}(\mu_{1},\mu_{2})= {\rm ad}_{[\mu_{1},\mu_{2}]_{\mathfrak k}}.
\end{equation}
Since any $G$-invariant tensor field on $M$ is $\widetilde{\nabla}$-parallel \cite[Proposition I.11]{Kow}, it follows that $\widetilde{\nabla}g = \widetilde{\nabla}\widetilde{T} = \widetilde{\nabla}\widetilde{R} = 0.$

Let ${\mathfrak g}_{0}\subset {\mathfrak k}$ be the subalgebra generated by all projections $[\mu_{1},\mu_{2}]_{\mathfrak k},$ $\mu_{1},\mu_{2}\in {\mathfrak m},$ and let $\widetilde{\mathfrak g}$ be the subalgebra $\widetilde{\mathfrak g} :={\mathfrak g}_{0}\oplus {\mathfrak m}$ of ${\mathfrak g},$ known as the {\em transvection algebra}. Denote by $G_{0}$ and $\widetilde{G}$ the connected subgroups of $G$ with Lie algebra ${\mathfrak g}_{0}$ and $\widetilde{\mathfrak g},$ respectively. We say that $\widetilde{G}$ is the {\em transvection group} of the {\em affinely manifold} $(M,\widetilde{\nabla})$ \cite[Ch. I]{Kow}. Then, because ${\mathfrak g}_{0}$ coincides precisely with the algebra generated by all curvature transformation $\widetilde{R}(\mu_{1},\mu_{2})$ on $T_{o}M\cong {\mathfrak m}$ \cite[Proposition I.2]{Kow}, $G_{o}$ is isomorphic to the {\em restricted holonomy group} of $(M,\widetilde{\nabla}).$ Therefore, if $G_{o}$ is assumed to be closed in $K,$ we obtain a new representation for $M$ as reductive space $M = \widetilde{G}/G_{o}$ with the same canonical connection.

We finish this section with the following result, which will be used later.
\begin{lemma}\label{lfirst} Let $(M_{1} = G_{1}/K_{1},\varphi_{1},\xi_{1},\eta_{1},g_{1})$ and $(M_{2} = G_{2}/K_{2},\varphi_{2},\xi_{2},\eta_{2},g_{2})$ be simply connected homogeneous almost contact metric manifolds with adapted reductive decompositions ${\mathfrak g}_{1} = {\mathfrak k}_{1}\oplus {\mathfrak m}_{1}$ and ${\mathfrak g}_{2} = {\mathfrak k}_{2} \oplus {\mathfrak m}_{2},$ respectively, and $o_{1}$ and $o_{2}$  the corresponding origin points. If $L\colon ({\mathfrak m}_{1},\langle\cdot,\cdot\rangle_{1})\to ({\mathfrak m}_{2},\langle\cdot,\cdot\rangle_{2})$ is a linear isometry satisfying
\begin{enumerate}
\item[{\rm (i)}] $L[\mu,\nu]_{{\mathfrak m}_{1}} = [L \mu,L \nu]_{{\mathfrak m}_{2}},$
\item[{\rm (ii)}] $L\circ {\rm ad}_{[\mu,\nu]_{{\mathfrak k}_{1}}} = {\rm ad}_{[L \mu,L \nu]_{{\mathfrak k}_{2}}}\circ L,$
\item[{\rm (iii)}] $L{\xi_{1}}_{o_{1}}= {\xi_{2}}_{o_{2}}$ and $L\circ\varphi_{1\,o_{1}} = \varphi_{2\,o_{2}}\circ L,$
\end{enumerate}
for all $\mu,\nu\in {\mathfrak m}_{1},$ then there exits a unique isometry $f$ of $(M_{1},g_{1})$ onto $(M_{2},g_{2})$ such that $f(o_{1}) = o_{2}$ and the differential of $f$ at $o_{1}$ coincides with $L.$ Moreover, the almost contact metric structures on $M_{1}$ and $M_{2}$ are isomorphic via $f.$
\end{lemma}
\begin{proof} Let $\widetilde{T}_{i}$ be the torsion and let $\widetilde{R}_{i}$ be the curvature of the canonical connection $\widetilde{\nabla}_{i}$ on $M_{i},$ $i = 1,2.$ From (i) and (ii), using (\ref{TR}), $L$ maps $\widetilde{T}_{1}$ and $\widetilde{R}_{1}$ at $o_{1}$ into $\widetilde{T}_{2}$ and $\widetilde{R}_{2}$ at $o_{2},$ respectively. Since $\widetilde{\nabla}_{i}\widetilde{T}_{i} = \widetilde{\nabla}_{i}\widetilde{R}_{i} = 0,$ it follows from \cite[Theorem 7.8]{KNI}, that there exists a unique affine isomorphism $f$ of $M_{1}$ onto $M_{2}$ such that $f(o_{1}) = o_{2}$ and $f_{*o_{1}} = L.$ Moreover, taking into account that $\widetilde{\nabla}_{1}$ and $\widetilde{\nabla}_{2}$ are metric connections and $L$ is a linear isometry, $f$ is an isometry.

Next, we show the last part of the lemma. Because $\xi_{1}$ and $\varphi_{1}$ are $G_{1}$-invariant, they are $\widetilde{\nabla}_{1}$-parallel. Hence, using that $f$ is an affine isomorphism, $f_{*}\xi_{1}$ and $f_{*}\circ\varphi_{1}\circ f_{*}^{-1}$ are $\widetilde{\nabla}_{2}$-parallel tensor fields on $M_{2}.$ According to \cite[Proposition I.38]{Kow}, they are then $\widetilde{G}_{2}$-invariant, where $\widetilde{G}_{2}$ denotes the transvection group of $(M_{2}=G_{2}/K_{2},\widetilde{\nabla}_{2}).$ Since $\xi_{2}$ and $\varphi_{2}$ are also $\widetilde{G}_{2}$-invariant and, from (iii), $f_{*\,o_{1}}\xi_{1\,o_{1}} = \xi_{2\,o_{2}}$ and $(f_{*}\circ \varphi_{1}\circ f_{*}^{-1})_{*\,o_{2}} = \varphi_{2\,o_{2}},$ it follows, using the transitivity of $\widetilde{G}_{2}$ on $M_{2}$ as subgroup of $G_{2},$ that $f_{*}\xi_{1} = \xi_{2}$ and $f_{*}\circ\varphi_{1}\circ f_{*}^{-1} = \varphi_{2}.$ This proves the result.
\end{proof}
\begin{remark}{\rm The quaterna $({\mathfrak m}_{i},\widetilde{T}_{i},\widetilde{R}_{i},\langle\cdot,\cdot\rangle_{i} = g_{i\,o_{i}}),$ for each $i = 1,2,$ is an} infinitesimal model {\rm associated to the reductive homogeneous manifold $(M_{i}= G_{i}/K_{i},g_{i}),$ known as the} canonical infinitesimal model, {\rm and the linear isometry $L$ satisfying (i) and (ii) is said to be an} isomorphism {\rm of infinitesimal models (see \cite{GGV} and \cite{TV}). Then Lemma \ref{lfirst} says that two simply connected homogeneous almost contact metric manifolds are isomorphic if there exists an isomorphism between the corresponding canonical infinitesimal models that preserves the almost contact metric structures.}
\end{remark}

\section{The tangent sphere bundle of a compact rank-one symmetric space} 

\setcounter{equation}{0}

 Let $G/K$ be a compact rank-one symmetric space and let ${\mathfrak g} = {\mathfrak k}\oplus{\mathfrak m}$ be the associated Cartan decomposition. Then, because $G$ is compact and semisimple, $G/K$ is equipped with an $G$-invariant metric determined by the restriction to ${\mathfrak m}$ of an ${\rm Ad}(G)$-invariant inner product $\langle\cdot,\cdot\rangle$ of ${\mathfrak g}.$ 
 
 Fixed an (one-dimensional) Cartan subspace ${\mathfrak a}$ of ${\mathfrak m},$ it follows from \cite[Lemma 3.2]{JC} that the subset $\Sigma^{+}\subset\Sigma$ of positive restricted roots of the triple $({\mathfrak g},{\mathfrak k},{\mathfrak a})$ is either $\{\varepsilon\},$ if $G/K\in \{{\mathbb S}^{n},{\mathbb R}{\bf P}^{n}\;(n\geq 2)\},$ or $\{\varepsilon,\frac{1}{2}\varepsilon\}$ for the rest of compact rank-one symmetric spaces, $\varepsilon$ being a linear function $\varepsilon\in ({\mathfrak a}^{\mathbb C})^{*}.$ (For the general theory of restricted roots, see \cite[Ch. III and Ch. VII]{He} and \cite[Section V.2]{Loos}.) For each $\lambda\in \Sigma^{+},$ put
  \begin{equation}\label{mk}
 {\mathfrak m}_{\lambda}  =  \{\mu\in {\mathfrak m}\colon {\rm ad}^{2}_{X}\mu = \lambda^{2}(X)\mu \},\quad 
 {\mathfrak k}_{\lambda}  =  \{\nu\in {\mathfrak k}\colon {\rm ad}^{2}_{X}\nu = \lambda^{2}(X)\nu\},
 \end{equation}
 $X$ being a generator of ${\mathfrak a}.$ Then, using \cite[Ch. VII, Lemma 11.3]{He}, we have the direct and $\langle\cdot,\cdot\rangle$-orthogonal decompositions 
 \[
 {\mathfrak m} = {\mathfrak a}\oplus{\mathfrak m}_{\varepsilon}\oplus{\mathfrak m}_{\varepsilon/2}, \quad {\mathfrak k}= {\mathfrak h}\oplus {\mathfrak k}_{\varepsilon}\oplus{\mathfrak k}_{\varepsilon/2},
 \]
 where ${\mathfrak h}$ is the centralizer of ${\mathfrak a}$ in ${\mathfrak k},$ that is the Lie algebra ${\mathfrak h} = \{\nu\in {\mathfrak k}\colon [{\mathfrak a},\nu] = 0\}$ of the subgroup $H$ of $K.$ Moreover, $\dim {\mathfrak m}_{\varepsilon} = \dim {\mathfrak k}_{\varepsilon} = m_{\varepsilon}$ and $\dim{\mathfrak m}_{\varepsilon/2} = \dim{\mathfrak k}_{\varepsilon/2}= m_{\varepsilon/2},$ where $m_{\varepsilon}$ and $m_{\varepsilon/2}$ are the multiplicities of the corresponding restricted roots $\varepsilon$ and $\varepsilon/2$ (see Table I). Here and in the sequels, we put ${\mathfrak m}_{\varepsilon/2} = {\mathfrak k}_{\varepsilon/2} = 0$ if $\varepsilon/2$ is not in $\Sigma.$  
 
 Applying \cite[Ch. VII, Lemma 11.4]{He}, it is immediate the following.
\begin{lemma}\label{lbrack1} We have:
\begin{equation}\label{brack1}
\begin{array}{l}
[{\mathfrak h},{\mathfrak m}_{\lambda}]\subset {\mathfrak m}_{\lambda},\quad [{\mathfrak h},{\mathfrak k}_{\lambda}]\subset{\mathfrak k}_{\lambda},\quad [{\mathfrak a},{\mathfrak m}_{\lambda}]\subset{\mathfrak k}_{\lambda},\quad [{\mathfrak a},{\mathfrak k}_{\lambda}]\subset {\mathfrak m}_{\lambda},\;\; \lambda = \varepsilon,\varepsilon/2,\\[0.4pc]
[{\mathfrak m}_{\varepsilon},{\mathfrak m}_{\varepsilon}]\subset {\mathfrak h},\quad [{\mathfrak m}_{\varepsilon},{\mathfrak m}_{\varepsilon/2}]\subset {\mathfrak k}_{\varepsilon/2},\quad [{\mathfrak m}_{\varepsilon},{\mathfrak k}_{\varepsilon}]\subset {\mathfrak a},\quad [{\mathfrak m}_{\varepsilon},{\mathfrak k}_{\varepsilon/2}]\subset {\mathfrak m}_{\varepsilon/2},\\[0.4pc]
[{\mathfrak m}_{\varepsilon/2},{\mathfrak m}_{\varepsilon/2}]\subset {\mathfrak h}\oplus {\mathfrak k}_{\varepsilon},\quad [{\mathfrak m}_{\varepsilon/2},{\mathfrak k}_{\varepsilon}]\subset{\mathfrak m}_{\varepsilon/2},\quad [{\mathfrak m}_{\varepsilon/2},{\mathfrak k}_{\varepsilon/2}]\subset {\mathfrak a}\oplus {\mathfrak m}_{\varepsilon},\\[0.4pc]
[{\mathfrak k}_{\varepsilon},{\mathfrak k}_{\varepsilon}]\subset {\mathfrak h},\quad [{\mathfrak k}_{\varepsilon},{\mathfrak k}_{\varepsilon/2}]\subset {\mathfrak k}_{\varepsilon/2},\quad [{\mathfrak k}_{\varepsilon/2},{\mathfrak k}_{\varepsilon/2}]\subset {\mathfrak h}\oplus {\mathfrak k}_{\varepsilon}.
\end{array}
\end{equation}
\end{lemma}

From \cite[Ch. VII, Lemma 2.3]{He}, for any $\mu_{\lambda}\in {\mathfrak m}_{\lambda},$ $\lambda\in\Sigma^+$, there exists a unique vector $\nu_{\lambda}\in{\mathfrak k}_\lambda$ such that $[u,\mu_{\lambda}] = \sqrt{-1}\lambda(u)\nu_{\lambda}$ and $[u,\nu_{\lambda}] = -\sqrt{-1}\lambda(u)\mu_{\lambda},$ for all $u\in {\mathfrak a}.$ Then,
\begin{equation}\label{eq.ms4.3}
[u,\mu_\lambda]=  -\lambda_{\mathbb R}(u)\nu_\lambda,
\quad [u,\nu_\lambda]= \lambda_{\mathbb R}(u)\mu_\lambda,
\end{equation}
where $\lambda_{\mathbb R}$ is the linear function $\lambda_{\mathbb R} \colon {\mathfrak a}\to{\mathbb R}$,
$\lambda\in\Sigma^+$, defined by the relation $\sqrt{-1}\lambda_{\mathbb R}=\lambda$.  
 Then fixing in ${\mathfrak m}_{\varepsilon}$ and ${\mathfrak m}_{\varepsilon/2}$ some $\langle\cdot,\cdot\rangle$-orthonormal basis $\{\mu^{j}_{\varepsilon}\}$ and $\{\mu^{p}_{\varepsilon/2}\},$ $j = 1,\dots, m_{\varepsilon},$ $p = 1,\dots, m_{\varepsilon/2},$ we take the unique bases $\{\nu^{j}_{\varepsilon}\}$ and $\{\nu^{p}_{\varepsilon/2}\}$ of ${\mathfrak k}_{\varepsilon}$ and ${\mathfrak k}_{\varepsilon/2}$ satisfying (\ref{eq.ms4.3}) for $\lambda = \varepsilon$ and $\lambda = \varepsilon/2.$ Clearly these last two basis are also $\langle\cdot,\cdot\rangle$-orthonormal. 
 
 Consider $X$ the unique (basis) vector in ${\mathfrak a}$ such that $\varepsilon_{\mathbb R}(X) = 1.$ Then, multiplying the inner product $\langle\cdot,\cdot\rangle$ by a positive constant, we can assume that $\langle X,X\rangle = 1.$ 
 \begin{lemma}\label{pbrack} The following identities hold:
 $$
 \begin{array}{l}
 [\mu^{j}_{\varepsilon},\mu^{k}_{\varepsilon}] = [\nu^{j}_{\varepsilon},\nu^{k}_{\varepsilon}],\quad [\mu^{j}_{\varepsilon},\nu^{k}_{\varepsilon}] = -[\nu^{j}_{\varepsilon},\mu^{k}_{\varepsilon}] =-\delta_{jk}X,\\[0.4pc]
 [\mu^{j}_{\varepsilon},\mu^{p}_{\varepsilon/2}] = [\nu^{j}_{\varepsilon},\nu^{p}_{\varepsilon/2}],\quad [\nu^{j}_{\varepsilon},\mu^{p}_{\varepsilon/2}] = -[\mu^{j}_{\varepsilon},\nu^{p}_{\varepsilon/2}],\\[0.4pc]
 [\mu^{p}_{\varepsilon/2},\mu^{q}_{\varepsilon/2}]_{{\mathfrak k}_{\varepsilon}} = -[\nu^{p}_{\varepsilon/2},\nu^{q}_{\varepsilon/2}]_{{\mathfrak k}_{\varepsilon}}= \frac{1}{2}([\mu^{p}_{\varepsilon/2},\mu^{q}_{\varepsilon/2}]-[\nu^{p}_{\varepsilon/2},\nu^{q}_{\varepsilon/2}]),\\[0.4pc]
 [\mu^{p}_{\varepsilon/2},\nu^{q}_{\varepsilon/2}]_{{\mathfrak m}_{\varepsilon}} = [\nu^{p}_{\varepsilon/2},\mu^{q}_{\varepsilon/2}]_{{\mathfrak m}_{\varepsilon}} = \frac{1}{2}([\mu^{p}_{\varepsilon/2},\nu^{q}_{\varepsilon/2}] + [\nu^{p}_{\varepsilon/2},\mu^{q}_{\varepsilon/2}]),\\[0.4pc]
 [\mu^{p}_{\varepsilon/2},\nu^{q}_{\varepsilon/2}]_{\mathfrak a} = -\frac{\delta_{pq}}{2}X,
\end{array}
$$
for all $j,k = 1,\dots, m_{\varepsilon}$ and $p,q = 1,\dots, m_{\varepsilon/2}.$
\end{lemma}  
 \begin{proof} Using that ${\rm ad}_{X}\colon {\mathfrak g}\to {\mathfrak g}$ is a Lie derivation, one gets
 \[
 {\rm ad}^{2}_{X}[U,V] = [{\rm ad}^{2}_{X}U,V] + [U,{\rm ad}_{X}^{2}V] + 2[{\rm ad}_{X}U,{\rm ad}_{X}V],\quad U,V\in {\mathfrak g}.
 \]
 From here, applying (\ref{mk}), (\ref{brack1}) and (\ref{eq.ms4.3}) these equalities are proved.
 \end{proof}
 
 Next we analyze the tangent bundle $T(G/K)$ and the tangent sphere bundle $T_{r}(G/K),$ of radius $r>0,$ of $G/K.$ 
 
 Let us first observe that on the trivial vector bundle $G\times{\mathfrak m}$ there are two Lie group actions that commute with each other: the left $G$-action $l_b \colon (a,x)\mapsto (ba,x)$ and the right $K$-action $r_k \colon (a,x)\mapsto (ak,{\rm Ad}_{k^{-1}}x).$ Let $\pi \colon G\times {\mathfrak m}\to G\times_K {\mathfrak m},$ $(a,x)\mapsto [(a,x)],$ be the natural projection for this right $K$-action. Then $\pi$ is $G$-equivariant and the mapping $\phi$ given by
\begin{equation}\label{eq.phi}
\phi \colon G\times_K {\mathfrak m}\to T(G/K),
\quad
[(a,x)]\mapsto (\tau_{a})_{*o}x,
\end{equation}
is a $G$-equivariant diffeomorphism which allows us to identify $T(G/K)$ with $G\times_{K}{\mathfrak m}.$ Then, taking into account that each translation $\tau_{a}$ is an isometry of $(G/K,g),$ it follows from (\ref{eq.phi}) that
\begin{equation}\label{TrGK}
T_{r}(G/K) = \phi(G\times_{K}{\mathcal S}_{\mathfrak m}(r)),
\end{equation} 
where ${\mathcal S}_{\mathfrak m}(r) = \{x\in {\mathfrak m}\colon \langle x,x\rangle = r^{2}\}$ is the sphere of radius $r>0$ in the Euclidean space $({\mathfrak m},\langle\cdot,\cdot\rangle).$  

On the other hand, let $W$ be the Weyl chamber in ${\mathfrak a}$ containing $X,$ 
\[
W = \{ w\in {\mathfrak a}\colon \varepsilon_{\mathbb R}(w)>0\} = \{tX\;:\; t\in {\mathbb R}^{+}\},
\]
which is naturally identified with ${\mathbb R}^{+}.$ Since each nonzero ${\rm Ad}(K)$-orbit in ${\mathfrak m}$ intersects ${\mathfrak a}$ and also $W,$ it follows that ${\mathfrak m}^{R} = {\mathfrak m}\setminus \{0\},$ where ${\mathfrak m}^{R}$ denotes the (open dense) subset of {\em regular points} ${\mathfrak m}^{R}:= {\rm Ad}(K)(W)$ of ${\mathfrak m}.$ Hence, we have:
\begin{enumerate}
\item[{\rm (i)}] the ${\rm Ad}(K)$-action on the unit sphere ${\mathcal S}_{\mathfrak m}(1)$ of ${\mathfrak m}$ is transitive, or equivalently,
\begin{equation}\label{SAd}
{\mathcal S}_{\mathfrak m}(r) = {\rm Ad}(K)rX,\quad r>0.
\end{equation}
\item[{\rm (ii)}] the mapping
\begin{equation}\label{f+}
\chi\colon G/H\times {\mathbb R}^{+}\to G\times_{K}({\mathfrak m}\setminus\{0\}),\quad (aH,t)\mapsto [(a,tX)],
\end{equation}
 is a well-defined $G$-equivariant diffeomorphism.
\end{enumerate} 
 The  property (i) is a well-known characterization for the class of compact rank-one symmetric spaces and this implies that for the simply connected case all their geodesics are closed (see \cite{Be1}). From (\ref{eq.phi}) and (\ref{f+}), $G/H\times {\mathbb R}^{+}$ is identified with the punctured tangent bundle $D(G/K) = T(G/K)\setminus\{\mbox{\rm zero section}\}$ of $T(G/K)$ via the $G$-equivariant diffeomorphism $\phi\circ \chi$ and,
using (\ref{TrGK}) and (\ref{SAd}), $T_{r}(G/K)$ with $G/H,$ i.e.
  \[
 T_{r}(G/K) = (\phi\circ\chi)(G/H\times\{r\})\cong G/H.
 \]
 
 According with \cite{JC}, the decomposition $\overline{\mathfrak m}={\mathfrak a}\oplus {\mathfrak m}_{\varepsilon}\oplus {\mathfrak m}_{\varepsilon/2} \oplus {\mathfrak k}_{\varepsilon}\oplus {\mathfrak k}_{\varepsilon/2}$ is ${\rm Ad}(H)$-irreducible and moreover, for $G/K\in \{{\mathbb C}{\bf P}^{n}\;(n\geq 1)\},$ because ${\mathfrak a}$, ${\mathfrak m}_{\varepsilon}$ and ${\mathfrak k}_{\varepsilon}$ are one-dimensional, ${\rm Ad}(H)$ acts trivially on $\overline{\mathfrak n} : = {\mathfrak a}\oplus {\mathfrak m}_{\varepsilon}\oplus {\mathfrak k}_{\varepsilon}.$ Then we have the following properties:
\begin{enumerate}
\item[{\rm (i)}]  ${\mathfrak g} = {\mathfrak h}\oplus\overline{\mathfrak m}$ is a reductive decomposition associated to $G/H,$
\item[{\rm (ii)}] the tangent space $T_{o_{H}}(G/H)$ at the origin $o_{H} = \{H\}$ is identified with $\overline{\mathfrak m},$ and
\item[{\rm (iii)}] any $G$-invariant Riemannian metric ${\mathbf g} = {\bf g}^{a_{0},a_{\lambda},b_{\lambda}},$ $\lambda\in \Sigma^{+},$ on $G/H$ is determined by 
\begin{equation}\label{Bb}
{\mathbf g}_{o_{H}} = a_{0}\langle\cdot,\cdot\rangle_{\mathfrak a} + a_{\varepsilon}\langle\cdot,\cdot\rangle_{{\mathfrak m}_{\varepsilon}} +  b_{\varepsilon}\langle\cdot,\cdot\rangle_{{\mathfrak k}_{\varepsilon}}  + a_{\varepsilon/2}\langle\cdot,\cdot\rangle_{{\mathfrak m}_{\varepsilon/2}}    + b_{\varepsilon/2}\langle\cdot,\cdot\rangle_{{\mathfrak k}_{\varepsilon/2}},
\end{equation}
where $a_{0},a_{\lambda}$ and $b_{\lambda},$ for $\lambda\in \Sigma^{+},$ are positive constants, except for $G/K\in \{{\mathbb C}{\bf P}^{n}\;(n\geq 1)\},$ where
\begin{equation}\label{BbCP}
{\bf g}_{o_{H}} = \prec\cdot,\cdot\succ + a_{\varepsilon/2}\langle\cdot,\cdot\rangle_{{\mathfrak m}_{\varepsilon/2}}    + b_{\varepsilon/2}\langle\cdot,\cdot\rangle_{{\mathfrak k}_{\varepsilon/2}},
\end{equation}
$\prec\cdot,\cdot\succ$ being an arbitrary inner product on $\overline{\mathfrak n}.$
\end{enumerate}

\begin{definition}{\rm A $G$-invariant metric ${\bf g}$ on $G/H$ such that at the origin is expressed as in (\ref{Bb}) is said to be an} orthogonal metric.
\end{definition}
As we will see in the next section, there are many $G$-invariant contact metrics in $T_{r}{\mathbb C}{\bf P}^{n}$ that are not orthogonal.

Under the previous identification $\phi\circ\chi,$ the metric $\overline{\bf g} = r^{2}{\bf g} + dr^{2}$ on the me\-tric cone ${\mathcal C}(G/H)$ gives an $G$-invariant Riemannian metric, also denoted by $\overline{\bf g},$ on $D(G/K).$ Then, taking $T_{1}G/K\cong G/H\times \{1\}$ as submanifold of $G/H\times{\mathbb R}^{+}\cong D(G/K),$ using Corollary \ref{cclasses} and the characterizations described in Section \ref{Preliminares} for Sasakian, Sasakian-Einstein and $3$-Sasakian structures, the following result is proved.

\begin{theorem}\label{tmain} Each $G$-invariant contact metric, Sasakian, Sasakian-Einstein or $3$-Sasakian structure $(\xi,{\bf g})$ on the unit tangent sphere $T_{1}(G/K$) can be extended, respectively, to a unique $G$-invariant almost K\"ahler, K\"ahler, K\"ahler Ricci-flat or hyperK\"ahler structure on the metric punctured tangent bundle $(D(G/K),\overline{\bf g}).$ 
\end{theorem} 

   \begin{remark}\label{rex}{\rm From the definition of $\chi$ in (\ref{f+}), one gets
 \[
 \chi_{*(o_{H},t)}(X,0) = \pi_{*(e,tX)}(X,0),\quad t\in {\mathbb R}^{+}.
 \]
 Then, $\overline{\bf g}$ on $G\times_{K}({\mathfrak m}\setminus \{0\})$ satisfies 
 \[
 \overline{\bf g}_{[(e,tX)]}(\pi_{*(e,tX)}(X,0),\pi_{*(e,tX)}(X,0)) = \overline{\bf g}_{(o_{H},t)}((X,0),(X,0)) = t^{2}a_{0}.
 \] 
 Hence, $\lim_{t\to 0^{+}}\overline{\bf g}_{[(e,tX)]}(\pi_{*(e,tX)}(X,0),\pi_{*(e,tX)}(X,0)) = 0$ and so $\overline{\bf g}$ cannot be extended to a Riemannian metric on all $G\times_{K}{\mathfrak m}\cong T(G/K).$}

\end{remark}

 \section{Invariant contact metric structures}
 
 \setcounter{equation}{0} 
 
In order to obtain all $G$-invariant contact metric structures on $T_{r}(G/K),$ we analyze the space of $G$-invariant vector fields on $T_{r}(G/K) = G/H$ or, equivalently, the subspace ${\rm Inv}(\overline{\mathfrak m})\subset \overline{\mathfrak m}$ of ${\rm Ad}(H)$-invariant vectors of $\overline{\mathfrak m},$ as candidates to be characteristic vector fields. This space contains the standard vector field $\xi^{S}$ on $T_{r}(G/K),$ for any compact rank-one symmetric space $G/K.$ In fact $\xi^{S}$ is the $G$-invariant vector field on $G/H\times {\mathbb R}^{+}$ such that $\xi^{S}_{(o_{H},t)} = (X,0),$ for all $t\in {\mathbb R}^{+},$ \cite[Proposition 4.8]{JC}. 

Two cases for $G/K$ are distinguished depending on whether $\dim {\rm Inv}(\overline{\mathfrak m})$ is one or greater than one.

\begin{theorem}\label{maingeneral1} If $G/K\in \{{\mathbb S}^{n},{\mathbb R}{\bf P}^{n}(n\geq 3),{\mathbb H}{\bf P}^{n} (n\geq 1), {\mathbb C}a{\bf P}^{2}\},$ any $G$-invariant contact metric structure on the tangent sphere bundle $T_{r}(G/K),$ $r>0,$ is given by the pairs $(\xi,{\bf g}),$ where $\xi =\pm \frac{1}{\kappa}\xi^{S},$ for some $\kappa>0,$ i.e. $\xi_{o_{H}} = \pm\frac{1}{\kappa}X,$ and the metric ${\bf g}$ is determined by
  \begin{equation}\label{gc}
 {\bf g}_{o_{H}} = \kappa^{2}\langle\cdot,\cdot\rangle_{\mathfrak a} + \frac{\kappa}{2q_{\varepsilon}}\langle\cdot,\cdot\rangle_{{\mathfrak m}_{\varepsilon}} + \frac{\kappa q_{\varepsilon}}{2}\langle\cdot,\cdot\rangle_{{\mathfrak k}_{\varepsilon}} + \frac{\kappa}{4q_{\varepsilon/2}}\langle\cdot,\cdot\rangle_{{\mathfrak m}_{\varepsilon/2}}  + \frac{\kappa q_{\varepsilon/2}}{4}\langle\cdot,\cdot\rangle_{{\mathfrak k }_{\varepsilon/2}},
 \end{equation}
 for some positive constants $q_{\lambda},$ $\lambda\in \Sigma^{+}.$ Moreover, $(\xi,{\bf g})$ is K-contact, indeed Sasakian, if and only if $q_{\lambda} = 1.$ 
 \end{theorem}
 
 \begin{proof} Let ${\bf g}$ be any $G$-invariant Riemannian metric ${\mathbf g} = {\bf g}^{a_{0},a_{\lambda},b_{\lambda}},$ $\lambda\in \Sigma^{+},$ on $G/H$ defined in (\ref{Bb}) and let $\xi$ be any unit $G$-invariant vector field on $(G/H,{\bf g}).$ Then $\xi_{o_{H}}$ in $\overline{\mathfrak m}$ is ${\rm Ad}(H)$-invariant and, therefore, the subspace ${\mathbb R}\xi_{o_{H}}$ must be ${\rm Ad}(H)$-irreducible. According with Table I, taking into account that the decomposition $\overline{\mathfrak m} = {\mathfrak a}\oplus {\mathfrak m}_{\varepsilon} \oplus {\mathfrak k}_{\varepsilon}\oplus {\mathfrak m}_{\varepsilon/2} \oplus {\mathfrak k}_{\varepsilon/2}$ is ${\rm Ad}(H)$-irreducible and the dimensions of the subspaces ${\mathfrak m}_{\varepsilon},$ ${\mathfrak k}_{\varepsilon},$ ${\mathfrak m}_{\varepsilon/2}$ and ${\mathfrak k}_{\varepsilon/2}$ are different from one, we have that ${\rm Inv}(\overline{\mathfrak m}) = {\mathbb R}X.$  Then $\xi = \pm \frac{1}{\kappa}\xi^{S}$ and the result has already been proven in \cite[Theorem 6.1]{JC1}.
\end{proof}
 
 If $G/K\in\{{\mathbb C}{\bf P}^{n}\,(n\geq 1)\}$ then ${\rm Inv}(\overline{\mathfrak m}) = \overline{\mathfrak n}$ and any $G$-invariant Riemannian metric ${\bf g}$ on $T_{r}{\mathbb C}{\bf P}^{n} = G/H$ is expressed at the origin as in (\ref{BbCP}). Put 
  $$
 \begin{array}{l}
 a_{0} = \prec X,X\succ,\quad a_{\varepsilon} = \prec\mu_{\varepsilon},\mu_{\varepsilon}\succ,\quad b_{\varepsilon} = \prec\nu_{\varepsilon},\nu_{\varepsilon}\succ,\\[0.4pc]
 a_{0\varepsilon} = \prec X,\mu_{\varepsilon}\succ,\quad b_{0\varepsilon} = \prec X,\nu_{\varepsilon}\succ,\quad c_{\varepsilon} = \prec \mu_{\varepsilon},\nu_{\varepsilon}\succ,
 \end{array}
 $$
 where $\{X,\mu_{\varepsilon},\nu_{\varepsilon}\}$ is an orthonormal basis of $(\overline{\mathfrak n},\langle\cdot,\cdot\rangle)$ satisfying (\ref{eq.ms4.3}). Then ${\bf g}$ is orthogonal if and only if $a_{0\varepsilon} = b_{0\varepsilon} = c_{\varepsilon} = 0.$ Next, consider the (unique) orthonormal basis $\{\xi_{1},\xi_{2},\xi_{3}\}$ of $(\overline{\mathfrak n},\prec\cdot,\cdot\succ)$ satisfying the following conditions:
\begin{enumerate}
\item[{\rm (i)}] $\xi_{1} = \frac{1}{\rho}X,$ for some $\rho >0,$ and $\xi_{2}\in {\mathbb R}\{X,\mu_{\varepsilon}\},$
\item[{\rm (ii)}] the bases $\{\xi_{1},\xi_{2}\}$ and $\{\xi_{1},\xi_{2},\xi_{3}\}$ have, respectively, the same orientation than $\{X,\mu_{\varepsilon}\}$ and $\{X,\mu_{\varepsilon},\nu_{\varepsilon}\}.$
\end{enumerate} 
 Therefore the matrix $A = (a_{ij})$ such that $\xi_{j} = a_{j1}X + a_{j2}\mu_{\varepsilon} + a_{j3}\nu_{\varepsilon},$ $j = 1,2,3,$ determines the inner product $\prec\cdot,\cdot\succ$ and it is expressed as
 \begin{equation}\label{A}
 A = 
 \left (
 \begin{array}{ccc}
 \frac{1}{\rho} & 0 & 0\\[0.4pc]
 a_{21} & a_{22} & 0\\[0.4pc]
 a_{31}& a_{32} & a_{33}
 \end{array}
 \right),
\end{equation}
where $a_{22}$ and $a_{33}$ are positive numbers. Note that ${\bf g}$ is orthogonal if and only if $A$ is diagonal.
\begin{remark}{\rm  From \cite[Proposition 4.8]{JC}, the induced metric $\widetilde{g}^{S}$ of the Sasaki metric $g^{S}$ on $T_{r}(G/K)$ is $G$-invariant, determined by 
\begin{equation}\label{gsasaki}
\widetilde{g}^{S}_{o_{H}} = \left\{
\begin{array}{lcl}
\langle\cdot,\cdot\rangle_{\mathfrak m} + r^{2}\langle\cdot,\cdot\rangle_{{\mathfrak k}_{\varepsilon}}, &\mbox{if} & G/K = {\mathbb S}^{n}\;\mbox{or}\;\; {\mathbb R}{\mathbf P}^{n},\\[0.4pc]
 \langle\cdot,\cdot\rangle_{\mathfrak m} + r^{2}\langle\cdot,\cdot\rangle_{{\mathfrak k}_{\varepsilon}} + \frac{r^{2}}{4}\langle\cdot,\cdot\rangle_{{\mathfrak k}_{\varepsilon/2}}, & \mbox{if} & G/K =  {\mathbb C}{\mathbf P}^{n},\; {\mathbb H}{\mathbf P}^{n}\;\mbox{or}\; \;{\mathbb C}a{\mathbf P}^{n}.
 \end{array}
 \right.
\end{equation}
Hence $\widetilde{g}^{S}$ is an orthogonal metric, even in $T_{r}{\mathbb C}{\bf P}^{n}.$ Here $a_{0} = a_{\varepsilon} = 1$ and $b_{\varepsilon} = r^{2},$ and $A$ is the diagonal matrix $(1,1,\frac{1}{r}).$}
\end{remark}
If moreover ${\bf g}$ is contact metric then its characteristic vector field $\xi$ is a $G$-invariant field with $\xi_{o_{H}}\in \overline{\mathfrak n}.$ We say that a $G$-invariant contact metric structure $(\xi,{\bf g})$ on $T_{r}{\mathbb C}{\bf P}^{n}$ is of
\begin{enumerate}
\item[$\bullet$] Type AI, AII or AIII, if $\xi_{o_{H}} = \pm \xi_{1},$ $\pm \xi_{2}$ or $\pm\xi_{3},$
\item[$\bullet$] Type BI, BII or BIII, if $\xi_{o_{H}} = \cos\theta\xi_{i} + \sin\theta\xi_{j},$ for $(i,j) = (1,2),$ $(1,3)$ or $(2,3),$ where $\theta\in ]0,2\pi[\setminus\{\pi/2,\pi,3\pi/2\},$ 
\item[$\bullet$] Type C, if $\xi_{o_{H}} = \cos\theta\cos\phi \xi_{1}+\sin\theta\cos\phi\xi_{2}+\sin\phi\xi_{3},$ where $\theta\in ]0,2\pi [\setminus \{\pi/2,\pi,3\pi/2\}$ and $\phi\in ]-\pi/2,\pi/2[\setminus \{0\}.$
\end{enumerate}
 
 \begin{theorem}\label{maingeneralCP1} A $G$-invariant contact metric structure $(\xi,{\bf g})$ on $T_{r}{\mathbb C}{\bf P}^{n}\,(n\geq 1)$ is of 
 \begin{enumerate}
 \item[$\bullet$] {\rm Type AI} if and only if $\xi_{o_{H}} = \pm\frac{1}{\kappa}X$ and
 $$
 \begin{array}{l}
a_{0}= \kappa^{2},\quad a_{\varepsilon} = \frac{\kappa}{2q_{\varepsilon}},\quad b_{\varepsilon} = \frac{q_{\varepsilon}}{2\kappa}(4\alpha^{2} + \kappa^{2}),\\
c_{\varepsilon} = \alpha,\quad a_{0\varepsilon} = b_{0\varepsilon} = 0,\quad a_{\varepsilon/2} = \frac{\kappa}{4q_{\varepsilon/2}},\quad b_{\varepsilon/2} = \frac{\kappa q_{\varepsilon/2}}{4},
\end{array}
$$
for all $\kappa,q_{\varepsilon}, q_{\varepsilon/2}>0$ and $\alpha\in {\mathbb R};$
 \item[$\bullet$]{\rm Type AII} if and only if $\xi_{o_{H}} = \pm\frac{1}{\kappa}\mu_{\varepsilon}$ and
 $$
\begin{array}{l}
a_{0}= \frac{\kappa}{2q_{\varepsilon}},\quad a_{\varepsilon}= \kappa^{2},\quad b_{\varepsilon} = \frac{q_{\varepsilon}}{2\kappa}(4\alpha^{2} + \kappa^{2}),\\[0.4pc]
a_{0\varepsilon} = c_{\varepsilon}= 0,\quad b_{0\varepsilon} = \alpha, \quad a_{\varepsilon/2} = \frac{\kappa}{4q_{\varepsilon/2}},\quad b_{\varepsilon/2} = \frac{\kappa q_{\varepsilon/2}}{4},
 \end{array}
 $$ 
 for all $\kappa,q_{\varepsilon}, q_{\varepsilon/2}>0$ and $\alpha\in {\mathbb R};$
  \item[$\bullet$] {\rm Type AIII} if and only if $\xi = \pm\frac{1}{\kappa}\nu_{\varepsilon}$ and 
 $$
\begin{array}{l}
a_{0} = \frac{\kappa}{2q_{\varepsilon}},\quad a_{\varepsilon} = \frac{q_{\varepsilon}}{2\kappa}(4\alpha^{2} + \kappa^{2}),\quad b_{\varepsilon} = \kappa^{2},\\[0.4pc]
a_{0\varepsilon} = \alpha, \quad b_{0\varepsilon} = c_{\varepsilon}= 0,\quad a_{\varepsilon/2} = b_{\varepsilon/2} =\frac{\kappa}{4} ,
 \end{array}
 $$
 for all $\kappa, q_{\varepsilon}>0$ and $\alpha\in {\mathbb R};$ 
\item[$\bullet$]{\rm Type BI} if and only if $\xi_{o_{H}} = \frac{1}{\kappa}(\epsilon_{\cos\theta}\frac{\varrho}{q_{\varepsilon}}X + \sin\theta\mu_{\varepsilon}),$ where $\epsilon_{\cos\theta}$ denotes the sing of $\cos\theta$ and $\varrho = \sqrt{q_{\varepsilon}(1-q_{\varepsilon}\sin^{2}\theta)},$ and
$$
\begin{array}{l}
a_{0} = \frac{\kappa^{2}\varrho^{2}}{\cos^{2}\theta},\quad a_{\varepsilon} = \frac{\kappa^{2}}{\cos^{2}\theta}(\cos^{2}\theta + (1-q_{\varepsilon})^{2}\sin^{2}\theta),\quad b_{\varepsilon} = \frac{q_{\varepsilon}(\kappa^{4}\alpha^{2}\varrho^{2} + \sin^{2}\theta\cos^{2}\theta)}{4\varrho^{2}\sin^{2}\theta},\\[0.4pc]
a_{0\varepsilon} = -\epsilon_{\cos\theta}\frac{\kappa^{2}\varrho}{\cos^{2}\theta}(1-q_{\varepsilon})\sin\theta,\quad 
b_{0\varepsilon} = \epsilon_{\cos\theta}\frac{\alpha\kappa^{3}q_{\varepsilon}\varrho}{2\cos\theta},\quad
c_{\varepsilon} = -\frac{\alpha\kappa^{3}\varrho^{2}}{2\sin\theta\cos\theta}, \\[0.4pc]
 a_{\varepsilon/2} = \frac{\kappa\sqrt{q_{\varepsilon}}}{4q_{\varepsilon/2}},\quad b_{\varepsilon/2} = \frac{\kappa\sqrt{q_{\varepsilon}}q_{\varepsilon/2}}{4},
\end{array}
$$
for all $\kappa,q_{\varepsilon/2}>0,$ $q_{\varepsilon}\in]0,\frac{1}{\sin^{2}\theta}[$ and $\alpha\in {\mathbb R};$
\item[$\bullet$]{\rm Type BII} if and only if $\xi_{o_{H}} = \frac{q_{\varepsilon}}{2\kappa}\cos\theta X + \sin\theta(\alpha \mu_{\varepsilon} + \frac{1}{\kappa}\nu_{\varepsilon}),$ and 
$$
\begin{array}{l}
a_{0} = \frac{4\kappa^{2}q_{\varepsilon}^{2}}{\varrho^{2}},\quad a_{\varepsilon} = \frac{1}{q_{\varepsilon}^{2}\varrho^{2}}(\varrho^{2} + 16\alpha^{2}q_{\varepsilon}^{2}\kappa^{4}\tan^{2}\theta),\quad b_{\varepsilon} = \kappa^{2}(1 + \frac{(4-\varrho)^{2}}{\varrho^{2}}\tan^{2}\theta + \frac{\alpha^{2}}{q_{\varepsilon}^{2}}),\\[0.4pc]
a_{0\varepsilon} = \frac{8\alpha\kappa^{3}q_{\varepsilon}}{\varrho^{2}}\tan\theta,\quad b_{0\varepsilon} = \frac{2\kappa^{2}q_{\varepsilon}(4-\varrho)}{\varrho^{2}}\tan\theta,\quad c_{\varepsilon} = \frac{\kappa\alpha}{q_{\varepsilon}^{2}}(4\kappa^{2}q_{\varepsilon}^{2}(4-\varrho)\tan^{2}\theta -\varrho^{2}),\\[0.4pc]
a_{\varepsilon/2} = b_{\varepsilon/2} = \frac{\kappa\sqrt{\varrho}}{2\varrho},
\end{array}
$$
where $\varrho = q_{\varepsilon}^{2}\cos^{2}\theta + 4(\alpha^{2}\kappa^{2} + 1)\sin^{2}\theta,$ for all $\kappa,q_{\varepsilon}>0$ and $\alpha\in {\mathbb R}.$

\item[$\bullet$] {\rm Type BIII} if and only if $\xi_{o_{H}} = \frac{1}{\kappa q_{\varepsilon}}(\cos\theta\mu_{\varepsilon} + q_{\varepsilon}\sin\theta\nu_{\varepsilon}),$ and 
$$
\begin{array}{l}
a_{0} = \frac{q_{\varepsilon}^{2}}{4},\quad a_{\varepsilon} = \frac{\kappa^{2}q_{\varepsilon}^{2}}{\varrho^{2}}(1+ \alpha^{2}\varrho^{2}\tan^{2}\theta),\quad b_{\varepsilon} = \frac{\kappa^{2}}{\varrho^{2}}(\varrho^{2}\alpha^{2} + (\varrho^{2} + (1-q_{\varepsilon}^{2})^{2}\sin^{2}\theta\cos^{2}\theta)),\\[0.4pc]
a_{0\varepsilon} = -\frac{\alpha\kappa q_{\varepsilon}^{2}}{2}\tan\theta,\quad b_{0\varepsilon} = \frac{\alpha\kappa q_{\varepsilon}}{2}, \quad c_{\varepsilon} = -\frac{q_{\varepsilon}\kappa^{2}}{\varrho^{2}}(\alpha^{2}\varrho^{2}\tan\theta + (1-q_{\varepsilon}^{2})\sin\theta\cos\theta),\\[0.4pc]
a_{\varepsilon/2} = b_{\varepsilon/2} = \frac{\kappa q_{\varepsilon}\sqrt{\varrho}}{4\varrho},
\end{array}
$$
where $\varrho = \cos^{2}\theta + q_{\varepsilon}^{2}\sin^{2}\theta,$ for all $\kappa,q_{\varepsilon}>0$ and $\alpha\in {\mathbb R};$

\item[$\bullet$] {\rm Type C} if and only if
$$
\begin{array}{lcl}
 \xi_{o_{H}} & = & \frac{1}{q_{\varepsilon}\delta}(\delta\cos\phi(\cos\theta + q_{\varepsilon}\alpha\sin\theta) + q_{\varepsilon}(\alpha\gamma_{1}-\gamma_{2})\sin\phi)X \\[0.4pc]
 & & + (\delta\sin\theta\cos\phi + \gamma_{1})\mu_{\varepsilon} + \frac{1}{\kappa}\nu_{\varepsilon},\;\mbox{and}
 \end{array}
 $$
 $$
 \begin{array}{l}
 a_{0} = q_{\varepsilon}^{2},\quad a_{\varepsilon} = \frac{1+\alpha^{2}q_{\varepsilon}^{2}}{\delta^{2}},\quad b_{\varepsilon} = \frac{\kappa^{2}}{\delta^{2}}(\gamma_{1}^{2}+ q_{\varepsilon}^{2}\gamma_{2}^{2}+\delta^{2}),\\[0.4pc]
 a_{0\varepsilon} = -\frac{\alpha q_{\varepsilon}^{2}}{\delta},\quad b_{0\varepsilon} = \frac{\kappa q_{\varepsilon}^{2}\gamma_{2}}{\delta},\quad c_{\varepsilon} = -\frac{\kappa}{\delta^{2}}(\gamma_{1} + \alpha q_{\varepsilon}^{2}\gamma_{2}),\\[0.4pc]
 a_{\varepsilon/2}^{2} = b_{\varepsilon/2}^{2} = \frac{1}{16\delta^{2}}(q_{\varepsilon}^{2}\delta^{2}\cos\theta^{2}\cos^{2}\phi + \varrho),
 \end{array}
 $$ 
for all $\kappa,q_{\varepsilon}>0$ and $\alpha \in {\mathbb R}$ such that $\varrho\cos^{2}\theta\cos^{2}\phi\leq \kappa^{2},$ where $\varrho = \cos^{2}\phi(\sin\theta - q_{\varepsilon}\alpha\cos\theta)^{2} + 4q_{\varepsilon}^{2}\sin^{2}\phi,$ 
$$
 \begin{array}{lcl}
 \gamma_{1} & = & \frac{\cot\phi}{\kappa^{2}\beta}((1-\kappa^{2}\beta\delta)\sin\theta -q_{\varepsilon}\alpha\cos\theta),\\[0.4pc]
 \gamma_{2} & = & \frac{\cot\phi}{\kappa^{2}\beta q_{\varepsilon}}(q_{\varepsilon}\alpha\sin\theta + (\kappa^{2}\beta\delta -q_{\varepsilon}^{2}(\alpha^{2} + \delta^{2}))\cos\theta)
 \end{array}
 $$
 and $\delta$ is positive number such that $(q_{\varepsilon}\cos\theta\cos\phi)^{2}\delta^{2} -2\kappa^{2}\beta\delta + \varrho = 0,$ where $\beta^{2} = \frac{4q_{\varepsilon}^{2}}{\kappa^{2}}.$ 
   \end{enumerate}
 \end{theorem}
 
 When ${\bf g}$ is an orthogonal metric, $(\xi,{\bf g})$ is said to be an {\em orthogonal contact metric structure}.
 
 \begin{corollary}\label{cCP} Any $G$-invariant $K$-contact structure $(\xi,{\bf g})$ on $T_{r}{\mathbb C}{\bf P}^{n}$ $(n\geq 1)$ of {\rm Type A} is orthogonal.
 \end{corollary}
 \begin{proof} Because $\xi$ is Killing, it follows from (\ref{U}) and (\ref{Killing}) that
 \begin{equation}\label{K1}
 {\bf g}([\xi,\mu_{1}]_{\overline{\mathfrak m}},\mu_{2}) + {\bf g}([\xi,\mu_{2}]_{\overline{\mathfrak m}},\mu_{1}) = 0,\quad \mu_{1},\mu_{2}\in \overline{\mathfrak m}.
  \end{equation}
Now, using (\ref{eq.ms4.3}) and Theorem \ref{maingeneralCP1} and putting in (\ref{K1}) $\mu_{1} = \mu_{2} = \mu_{\varepsilon},$ if $(\xi,{\bf g})$ is of Type AI, and $\mu_{1} = \mu_{2} = X,$ if $(\xi,{\bf g})$ is of Type AII or AIII, we get $\alpha = 0.$ Then, from Theorem \ref{maingeneralCP1}, $(\xi,{\bf g})$ is orthogonal.
 \end{proof}

 \begin{theorem}\label{maingeneral} All orthogonal contact metric structures on $T_{r}{\mathbb C}{\bf P}^{n}$ and those that are also $K$-contact are described in {\rm Table II}. Moreover, each orthogonal contact metric structure of Type {\rm AI}, determined by a triple of positive constants $(\kappa,q_{\varepsilon},q_{\varepsilon/2}),$ is isomorphic to the corresponding of Type {\rm AII}.   
  
  
  \bigskip
\begin{tabular}{llllllll}
\multicolumn{8}{c}{\rm Table II. Orthogonal contact metric structures $(\xi,{\bf g}^{a_{0},a_{\varepsilon},b_{\varepsilon},a_{\varepsilon/2},b_{\varepsilon/2}})$ on $T_{r}{\mathbb C}{\bf P}^{n}$}\\
\hline\noalign{\smallskip}
{\rm Type} & $\xi_{o_{H}}$ & $a_{0}$
& $a_{\varepsilon}$
& $b_\varepsilon$
& $a_{\varepsilon/{\scriptscriptstyle 2}}$ & $b_{\varepsilon/{\scriptscriptstyle 2}}$ & $K$-{\rm contact}\\
\noalign{\smallskip}\hline\noalign{\smallskip}
{\rm AI}& ${\scriptstyle \pm\frac{1}{\kappa}X}$
& $\kappa^{2}$ & $\frac{\kappa}{2q_{\varepsilon}}$
& $\frac{\kappa q_{\varepsilon}}{2}$ & $\frac{\kappa}{4q_{\varepsilon/2}}$ & $\frac{\kappa q_{\varepsilon/2}}{4}$ & $\scriptstyle{q_{\varepsilon} = q_{\varepsilon/2} = 1}$ \\ \smallskip
{\rm AII} & ${\scriptstyle \pm\frac{1}{\kappa}\mu_{\varepsilon}}$
& $\frac{\kappa}{2q_{\varepsilon}}$ & $\kappa^{2}$
& $\frac{\kappa q_{\varepsilon}}{2}$ & $\frac{\kappa}{4q_{\varepsilon/2}}$ & $\frac{\kappa q_{\varepsilon/2}}{4}$  & $\scriptstyle{q_{\varepsilon} = q_{\varepsilon/2} = 1}$ \\
\smallskip
{\rm AIII} & ${\scriptstyle \pm\frac{1}{\kappa}\nu_{\varepsilon}}$
& $\frac{\kappa}{2q_{\varepsilon}}$& $\frac{\kappa q_{\varepsilon}}{2}$ & $\kappa^{2}$
& $\frac{\kappa}{4}$ & $\frac{\kappa}{4}$ & $\scriptstyle{q_{\varepsilon} = 1}$\\ \smallskip
{\rm BI} & ${\scriptstyle \frac{1}{\kappa}(\cos\theta X + \sin\theta \mu_{\varepsilon})}$ & $\kappa^{2}$ & $\kappa^{2}$ & $\frac{1}{4}$ & $\frac{\kappa}{4q_{\varepsilon/2}}$ & $\frac{\kappa q_{\varepsilon/2}}{4}$ & $\scriptstyle{\kappa = \frac{1}{2},\;q_{\varepsilon/2} = 1}$\\ \smallskip
{\rm BII} & ${\scriptstyle \frac{1}{\kappa} (\cos\theta X + \sin\theta \nu_{\varepsilon})}$ & $\kappa^{2}$ & $\frac{1}{4}$ & $\kappa^{2}$ & $\frac{\kappa}{4}$ & $\frac{\kappa}{4}$ & $\scriptstyle{\kappa = \frac{1}{2}}$\\ \smallskip
{\rm BIII} & ${\scriptstyle \frac{1}{\kappa}(\cos\theta \mu_{\varepsilon} + \sin\theta \nu_{\varepsilon})}$ & $\frac{1}{4}$ & $\kappa^{2}$ & $\kappa^{2}$ & $\frac{\kappa}{4}$ & $\frac{\kappa}{4}$ & $\scriptstyle{\kappa = \frac{1}{2}}$ \\ \smallskip
{\rm C} & ${\scriptstyle 2(\cos\theta\cos\phi X+\sin\theta\cos\phi\mu_{\varepsilon}+\sin\phi\nu_{\varepsilon})}$ & $\frac{1}{4}$ & $\frac{1}{4}$ & $\frac{1}{4}$ & $\frac{1}{8}$ & $\frac{1}{8}$ & $\scriptstyle{\rm Yes}$\\
\hline
\end{tabular}

\noindent \hspace{0.3cm} ${\scriptstyle \theta\in ]0,2\pi [\setminus \{\pi/2,\pi,3\pi/2\},\;\phi\in ]-\pi/2,\pi/2[\setminus \{0\}}$
\end{theorem}

\begin{remark}\label{r1}{\rm The tangent sphere bundles $T_{r}{\mathbb S}^{2}$ and $T_{r}{\mathbb R}{\bf P}^{2}$ are Lie groups naturally isomorphic to ${\rm SO}(3)\cong {\mathbb R}{\bf P}^{3}$ and ${\rm SO}(3)/\pm{\rm I}_{2},$ respectively. Since ${\mathbb S}^{2}\cong {\mathbb C}{\bf P}^{1}= {\rm SU}(2)/{\rm S}({\rm U}(1)\times{\rm U}(1)),$ invariant vector fields of $T_{r}{\mathbb S}^{2}$ and $T_{r}{\mathbb R}{\bf P}^{2}$ can be viewed as vectors of the Lie algebra ${\mathfrak s}{\mathfrak u}(2)\cong {\mathfrak s}{\mathfrak o}(3)$ and therefore, the set of their invariant contact metric structures are described in Theorem \ref{maingeneralCP1} for $n = 1,$ where ${\mathfrak m}_{\varepsilon/2} = {\mathfrak k}_{\varepsilon/2} = 0.$ }
 
 {\rm The $G$-invariant contact metric ${\bf g}$ on $T_{r}{\mathbb C}{\bf P}^{n} = G/H,$ $n\geq 1,$ such that ${\bf g}_{o_{H}} = \frac{1}{4}\langle\cdot,\cdot\rangle_{\overline{\mathfrak n}}+ \frac{1}{8}\langle\cdot,\cdot\rangle_{{\mathfrak m}_{\varepsilon/2}\oplus{\mathfrak k}_{\varepsilon/2}}$ defines $K$-contact metric structures of all types in Table II. Therefore the set of its associated (Killing) characteristic vector fields consists of all $G$-invariant vector fields $\xi$ such that $\xi_{o_{H}}$ belongs to the $2$-sphere ${\mathcal S}_{\overline{\mathfrak n}}(2)$ of radius $2$ in the Euclidean space $(\overline{\mathfrak n},\langle\cdot,\cdot\rangle).$ In Theorem \ref{t1} we prove that ${\bf g}$ is indeed a $3$-Sasakian metric.}
\end{remark}
From Theorems \ref{maingeneralCP1} and \ref{maingeneral} and Corollary \ref{cCP}, we have the following.
  \begin{corollary}\label{mainrank1} Any $G$-invariant Riemannian metric ${\bf g}$ on the tangent sphere bundle $T_{r}{\mathbb C}{\bf P}^{n}$ $(n\geq 1)$ such that the pair $(\xi = \frac{1}{\kappa}\xi^{S},{\bf g})$ is a contact metric structure for some constant $\kappa>0,$ is determined by
 \begin{equation}\label{Type A}
 \begin{array}{l}
a_{0}= \kappa^{2},\quad a_{\varepsilon} = \frac{\kappa}{2q_{\varepsilon}},\quad b_{\varepsilon} = \frac{q_{\varepsilon}}{2\kappa}(4\alpha^{2} + \kappa^{2}),\\
c_{\varepsilon} = \alpha,\quad a_{0\varepsilon} = b_{0\varepsilon} = 0,\quad a_{\varepsilon/2} = \frac{\kappa}{4q_{\varepsilon/2}},\quad b_{\varepsilon/2} = \frac{\kappa q_{\varepsilon/2}}{4},
\end{array}
\end{equation}
for some positive constants $q_{\varepsilon},$ $q_{\varepsilon/2}$ and $\alpha\in {\mathbb R}.$ Moreover, $(\xi,{\bf g})$ is $K$-contact, indeed Sasakian, if and only if $q_{\varepsilon} = q_{\varepsilon/2} = 1$ and $\alpha = 0,$ or equivalently 
 \[
 {\bf g}_{o_{H}} = \kappa^{2}\langle\cdot,\cdot\rangle_{\mathfrak a} + \frac{\kappa}{2}\langle\cdot,\cdot\rangle_{{\mathfrak m}_{\varepsilon}\oplus{\mathfrak k}_{\varepsilon}}+\frac{\kappa}{4}\langle\cdot,\cdot\rangle_{{\mathfrak m}_{\varepsilon/2}\oplus{\mathfrak k}_{\varepsilon/2}}.
 \]
 \end{corollary}
\begin{remark}{\rm Theorem 6.1 in \cite{JC1} must include the non-orthogonal contact metrics in $T_{r}{\mathbb C}{\bf P}^{n},$ described in (\ref{Type A}) for $\alpha\neq 0.$
}
\end{remark}

Before proceeding to the proof of Theorems \ref{maingeneralCP1} and \ref{maingeneral}, we first analyze, as a consequence of these theorems, the contact character of the induced Sasaki metric $\widetilde{g}^{S}$ on $T_{r}(G/K).$ 

For any Riemannian manifold $(M,g),$ the pair $(\frac{1}{2}\xi^{S},\frac{1}{4}\widetilde{g}^{S})$ is a contact metric structure on $T_{1}M$ (see for example \cite[Ch. 9]{Bl}) and Tashiro \cite{Tash} proved that it is $K$-contact, indeed Sasakian, if and only if $(M,g)$ has constant sectional curvature $1.$ Now we have the following version of Tashiro's Theorem for tangent sphere bundles $T_{r}(G/K)$ of any radius $r>0.$ The proof follows comparing expressions (\ref{gc}) and (\ref{gsasaki}) and using Theorems \ref{maingeneral1} and \ref{maingeneral}.
\begin{corollary} The induced Sasaki metric $\widetilde{g}^{S}$ on $T_{r}(G/K)$ is contact if and only if $r = \frac{1}{2},$ but it is not $K$-contact. On $T_{{\scriptstyle 1/2}}{\mathbb C}{\bf P}^{n}$ $(n\geq 1),$ $\widetilde{g}^{S}$ defines contact metric structures of Type {\rm AI} and {\rm BI}.

 For all $r>0,$ $\frac{1}{4r^{2}}\widetilde{g}^{S}$ is a contact metric on $T_{r}(G/K)$ and it is $K$-contact if and only if $r = 1$ and $G/K = {\mathbb S}^{n}$ or ${\mathbb R}{\bf P}^{n}.$ Then, $\frac{1}{4}\widetilde{g}^{S}$ is a Sasakian metric. On $T_{r}{\mathbb C}{\bf P}^{n},$ $\frac{1}{4r^{2}}\widetilde{g}^{S}$ defines contact metric structures of Type {\rm AI}, {\rm AII} and {\rm BI}.
 \end{corollary}

For the proof of Theorems \ref{maingeneralCP1} and \ref{maingeneral}, we will need to carry out a somewhat more detailed study of ${\mathbb C}{\bf P}^{n}$ as the symmetric space $G/K = {\rm SU}(n+1)/{\rm S}({\rm U}(1)\times {\rm U}(n))$ (see \cite[Ch. III and X]{He}). 

Consider the Lie algebra ${\mathfrak g} = {\mathfrak s}{\mathfrak u}(n+1)$ of ${\rm SU}(n+1)$ as compact real form of the complex simple Lie algebra ${\mathfrak a}_{n} = {\mathfrak s}{\mathfrak l}(n+1,{\mathbb C})$ of the complex $(n+1)\times(n+1)$ matrices of trace $0$ and denote by $E_{jk}$ the $(n+1)\times(n+1)$-matrix with $1$ in the entry of the $j$th row and the $k$th column and $0$ elsewhere. A Cartan subalgebra ${\mathfrak t}_{\mathbb C}$ of ${\mathfrak a}_{n}$ consists of the complex diagonal matrices of trace $0,$ generated by the set $\{T_{i} = E_{ii}-E_{i+1\,i+1},\;(1\leq i\leq n)\}.$ 

Let $e_{i},$ $i = 1,\dots, n+1,$ be the linear function on ${\mathfrak t}_{\mathbb C}$ such that $e_{i}(T)= T_{ii},$ for all $T\in {\mathfrak t}_{\mathbb C}.$ Then $\Pi = \{\alpha_{1},\dots ,\alpha_{n}\}$ is a system of simple roots of ${\mathfrak a}_{n},$ where $\alpha_{i} = e_{i}-e_{i+1},$ and we can take as set of positive roots $\Delta^{+}$ to $\Delta^{+} = \{\alpha_{pq} = \alpha_{p}+\dots + \alpha_{q}\;\colon 1\leq p \leq q \leq n\}.$ A root vector $E_{\alpha_{pq}}$ of $\alpha_{pq}$ is the matrix $E_{p\,q+1}.$ Hence, putting
 \begin{equation}\label{ABC}
A_{jk}  = \sqrt{-1}(E_{jj} -E_{kk}),\;\;\; B^{0}_{jk}  =   E_{jk} - E_{kj},\;\;\; B^{1}_{jk} = \sqrt{-1}(E_{jk} +E_{kj}),
\end{equation}
 the set of matrices $\{A_{i\,i+1}\;(1\leq i\leq n),\; B^{a}_{jk} \;(a = 0,1;\;1\leq j<k\leq n+1)\}$ is a basis of ${\mathfrak s}{\mathfrak u}(n+1).$ In what follows we shall need the Lie  multiplication table:
\begin{equation}\label{bracABC}
\begin{array}{lcl}
[A_{rj},A_{kl}] & = & 0,\\[0.4pc]
[A_{rj},B^{a}_{kl}] & = & (-1)^{a}\delta_{rk}B^{a+1}_{rl} - \delta_{rl}B^{a+1}_{rk} +(-1)^{a+1} \delta_{jk}B^{a+1}_{jl} + \delta_{jl} B^{a+1}_{jk},\\[0.4pc]
[B^{a}_{rj},B^{b}_{kl}] & = & -\delta_{rk}B^{a+b}_{jl}  + (-1)^{b}\delta_{rl}B^{a+b}_{jk} + (-1)^{a}\delta_{jk}B^{a+b}_{rl} + (-1)^{a+b+1}\delta_{jl}B^{a+b}_{rk},
\end{array}
\end{equation} 
for all $a,b\in \{0,1\},$ $a\leq b.$ 

If $n = 1,$ the Lie algebra ${\mathfrak k}$ of $K$ is one-dimensional, where ${\mathfrak k} = {\mathbb R} A_{12},$ and ${\mathfrak m} = {\mathbb R}\{B^{0}_{12},B^{1}_{12}\}.$ If $n\geq 2,$ a simple root system of ${\mathfrak k}$ is given by $\overline{\pi} =\{\alpha_{2},\dots ,\alpha_{n}\}$ and so $\overline{\Delta}^{+}(t) = \{\alpha_{pq}\colon 2\leq p\leq q \leq n\}$ is a positive root system. Hence,
$$
  \begin{array}{lcl}
 {\mathfrak k} & = & {\mathbb R}A_{12}\oplus (\sum_{j=2}^{n}{\mathbb R}A_{j\,j+1}\oplus \sum_{2\leq j<k\leq n+1}({\mathbb R}B^{0}_{jk} \oplus {\mathbb R}B^{1}_{jk}))\cong {\mathbb R}\oplus{\mathfrak s}{\mathfrak u}(n),\\[0.4pc]
  {\mathfrak m} & = & \sum_{i=1}^{n}({\mathbb R}B^{0}_{1\,i+1} \oplus {\mathbb R}B^{1}_{1\,i+1}).
  \end{array}
  $$
  Note that the canonical $G$-invariant complex structure on ${\mathbb C}{\bf P}^{n}$ is determined by the ${\rm Ad}(K)$-invariant endomorphism $I = {\rm ad}_{A_{12}\mid {\mathfrak m}}$ of ${\mathfrak m}.$ We choose as Cartan subspace to ${\mathfrak a} = {\mathbb R} X,$ where $X = \frac{1}{2}B^{0}_{12}.$ From (\ref{bracABC}), the centralizer ${\mathfrak h}$ of ${\mathfrak a}$ in ${\mathfrak k}$ is trivial if $n = 1.$ If $n = 2,$ ${\mathfrak h} = {\mathbb R}(A_{12} + 2A_{23})$ and, if $n\geq 3,$ it is given by
\begin{equation}\label{frakh}
{\mathfrak h} = {\mathbb R}(A_{12} + 2 A_{23}) \oplus \Big(\sum_{i= 3}^{n}{\mathbb R}A_{i\,i+1}\oplus \sum_{3\leq j<k\leq n+1}({\mathbb R}B^{0}_{jk}\oplus {\mathbb R}B^{1}_{jk})\Big)\cong {\mathbb R}\oplus {\mathfrak s}{\mathfrak u}(n-1).
\end{equation}
Moreover, we have
\[
{\rm ad}_{X}^{2}B^{1}_{12} = -B^{1}_{12},\quad {\rm ad}_{X}^{2}B^{a}_{1\,j+2} = -\frac{1}{4}B^{a}_{1\,j+2},\quad a = 0,1;\;\;j = 1,\dots, n-1.
\]
Hence, ${\mathfrak m}_{\varepsilon} = {\mathbb R}B^{1}_{12}$ and ${\mathfrak m}_{\varepsilon/2} = \sum_{j=1}^{n-1}({\mathbb R}B^{0}_{1\,j+2}\oplus{\mathbb R}B^{1}_{1\,j+2}).$ Now, taking the trace form $\langle U, V\rangle = -2\,{\rm trace}\, UV,$ the set $\{X,\mu_{\varepsilon},\mu_{\varepsilon}^{j,a}\;(a = 0,1;\,j = 1,\dots,n-1)\}$ is an orthonormal basis of $({\mathfrak m},\langle\cdot,\cdot\rangle_{\mid {\mathfrak m}}),$ where
\[
\mu_{\varepsilon} = \frac{1}{2}B^{1}_{12},\quad \mu^{j,a}_{\varepsilon/2} = \frac{1}{2}B^{a}_{1\,j+2}.
\]
Applying (\ref{eq.ms4.3}) and (\ref{bracABC}), one gets that ${\mathfrak k}_{\varepsilon} = {\mathbb R}\nu_{\varepsilon}$ and ${\mathfrak k}_{\varepsilon/2} = {\mathbb R}\{\nu^{j,a}_{\varepsilon/2}\;(a = 0,1;\,j = 1,\dots, n-1)\},$ where
\[
\nu_{\varepsilon} = -\frac{1}{2}A_{12},\quad \nu^{j,a}_{\varepsilon/2} = \frac{1}{2}B^{a}_{2\,j+2}.
\]
Moreover, the nonzero brackets of these basic elements of $\overline{\mathfrak m}$ satisfy the following equalities:
 \begin{equation}\label{brackTCP1}
 \begin{array}{l}
 [X,\mu_{\varepsilon}] = -\nu_{\varepsilon},\quad [X,\nu_{\varepsilon}] = \mu_{\varepsilon},\quad [X,\mu^{j,a}_{\varepsilon/2}] = -\frac{1}{2}\nu^{j,a}_{\varepsilon/2},\quad [X,\nu^{j,a}_{\varepsilon/2}] = \frac{1}{2}\mu^{j,a}_{\varepsilon/2},\\[0.5pc]
  [\mu_{\varepsilon},\mu^{j,a}_{\varepsilon/2}]  =  [\nu_{\varepsilon},\nu^{j,a}_{\varepsilon/2}] = \frac{(-1)^{a}}{2}\nu^{j,a+1}_{\varepsilon/2}, \quad [\mu_{\varepsilon}, \nu^{j,a}_{\varepsilon/2}] = -[\nu_{\varepsilon}, \mu^{j,a}_{\varepsilon/2}] = \frac{(-1)^{a}}{2}\mu^{j,a+1}_{\varepsilon/2},\\[0.5pc] 
  [\mu^{j,0}_{\varepsilon/2},\nu^{j,0}_{\varepsilon/2}]  =  [\mu^{j,1}_{\varepsilon/2},\nu^{j,1}_{\varepsilon/2}] = -\frac{1}{2}X, \quad[\mu^{j,0}_{\varepsilon/2}, \nu^{j,1}_{\varepsilon/2}] = -[\mu^{j,1}_{\varepsilon/2}, \nu^{j,0}_{\varepsilon/2}] = \frac{1}{2}\mu_{\varepsilon},\\[0.5pc] 
 [\mu_{\varepsilon},\nu_{\varepsilon}] = -X, \quad [\mu^{j,0}_{\varepsilon/2},\mu^{j,1}_{\varepsilon/2}]= -\frac{1}{2}\nu_{\varepsilon}+ \frac{1}{2}(\frac{1}{2}(A_{12} + 2A_{23}) + A_{3\,j+2})\in {\mathfrak k}_{\varepsilon} \oplus {\mathfrak h},\\[0.4pc]
 [\nu^{j,0}_{\varepsilon/2},\nu^{j,1}_{\varepsilon/2}] = \frac{1}{2}\nu_{\varepsilon}+ \frac{1}{2}(\frac{1}{2}(A_{12} + 2A_{23}) + A_{3\,j+2})\in {\mathfrak k}_{\varepsilon} \oplus {\mathfrak h},\\[0.4pc] 
 [\mu^{j,a}_{\varepsilon/2},\mu^{k,a}_{\varepsilon/2}] = [\nu^{j,a}_{\varepsilon/2},\nu^{k,a}_{\varepsilon/2}] = -\frac{1}{4}B^{0}_{j+2\,k+2}\in {\mathfrak h}\quad (j\neq k),\\[0.5pc]
 [\mu^{j,0}_{\varepsilon/2},\mu^{k,1}_{\varepsilon/2}]= 
[\nu^{j,0}_{\varepsilon/2},\nu^{k,1}_{\varepsilon/2}] =  -\frac{1}{4}B^{1}_{j+2\,k+2}\in {\mathfrak h}\quad (j\neq k).
\end{array}
\end{equation}
\begin{lemma}\label{lh} We have
\[
{\mathfrak h} = {\mathbb R}\{[u,v]_{\mathfrak h}\colon u,v\in \overline{\mathfrak m}\}.
\]
\end{lemma}
\begin{proof} From (\ref{brackTCP1}), we have
$$
\begin{array}{lcl}
A_{12} +2A_{23} =  4[\mu^{1,0}_{\varepsilon/2},\mu^{1,1}_{\varepsilon/2}]_{\mathfrak h}, & & A_{i\,i+1} = 2([\mu^{i-1,0}_{\varepsilon/2},\mu^{i-1,1}_{\varepsilon/2}]_{\mathfrak h} - [\mu^{i-2,0}_{\varepsilon/2},\mu^{i-2,1}_{\varepsilon/2}]_{\mathfrak h}),\\[0.4pc]
B^{0}_{jk}  =  -4[\mu^{j-2,0}_{\varepsilon/2},\mu^{k-2,0}_{\varepsilon/2}]_{\mathfrak h}, & & B^{1}_{jk}  =  -4[\mu^{j-2,0}_{\varepsilon/2},\mu^{k-2,1}_{\varepsilon/2}]_{\mathfrak h},
\end{array}
$$
for all $i = 3,\dots, n$ and $3\leq j<k\leq n+1.$ Then, this lemma follows from (\ref{frakh}).
\end{proof}

Finally, applying (\ref{bracABC}) and (\ref{frakh}), we determine the mapping ${\rm ad}_{\mathfrak h}\colon \overline{\mathfrak m}\to \overline{\mathfrak m}.$

\begin{lemma}\label{hbrackets} We have
$$
\begin{array}{l}
[{\mathfrak h},\overline{\mathfrak n}] = 0,\quad [A_{12} + 2A_{23},\mu^{j,a}_{\varepsilon/2}] = (-1)^{a}(\mu^{j,a+1}_{\varepsilon/2} + 2\delta_{1j} \mu^{1,a+1}_{\varepsilon/2}),\\[0.4pc]
[A_{12} + 2A_{23},\nu^{j,a}_{\varepsilon/2}] = (-1)^{a}(\nu^{j,a+1}_{\varepsilon/2} + 2\delta_{1j}\nu^{1,a+1}_{\varepsilon/2}),\\[0.4pc]
 [A_{i\,i+1},\mu^{j,a}_{\varepsilon/2}] = (-1)^{a}(\delta_{i\,j+1} \mu^{i-1,a+1}_{\varepsilon/2} - \delta_{i\,j+2} \mu^{i-2,a+1}_{\varepsilon/2}),\\[0.4pc]
 [A_{i\,i+1},\nu^{j,a}_{\varepsilon/2}] = (-1)^{a}(\delta_{i\,j+1} \nu^{i-1,a+1}_{\varepsilon/2} - \delta_{i\,j+2} \nu^{i-2,a+1}_{\varepsilon/2}),\\[0.4pc]
 [B^{a}_{pq},\mu^{j,b}_{\varepsilon/2}] = (-1)^{a(b+1)}(\delta_{q\,j+2}\mu^{p-2,a+b}_{\varepsilon/2} + (-1)^{a+1}\delta_{p\,j+2}\mu^{q-2,a}_{\varepsilon/2}),\\[0.4pc]
 [B^{a}_{pq},\nu^{j,b}_{\varepsilon/2}] = (-1)^{a(b+1)}(\delta_{q\,j+2}\nu^{p-2,a}_{\varepsilon/2} + (-1)^{a+1}\delta_{p\,j+2}\nu^{q-2,a}_{\varepsilon/2}),
 \end{array}
$$
for all $j = 1,\dots, n-1,$ $i = 3,\dots, n$ and $3\leq p<q\leq n+1.$
\end{lemma}

\noindent{\bf Proof of Theorem \ref{maingeneralCP1}.} From (\ref{A}), we get
\begin{equation}\label{e1}
\begin{array}{l}
a_{0}= \rho^{2},\quad a_{\varepsilon} = \frac{1}{a_{22}^{2}}(1 + a_{21}^{2}\rho^{2}),\quad b_{\varepsilon} = \frac{1}{a_{22}^{2}a_{33}^{2}}(\rho^{2}A_{13}^{2}+a_{22}^{2} + a_{32}^{2}),\\[0.4pc]
a_{0\varepsilon} = -\frac{\rho^{2}a_{21}}{a_{22}},\quad b_{0\varepsilon} = \frac{\rho^{2}}{a_{22}a_{33}}A_{13},\quad c_{\varepsilon} = -\frac{\rho^{2}}{a_{22}^{2}a_{33}}(a_{21}A_{13}+\frac{a_{32}}{\rho^{2}}),
\end{array}
\end{equation}
where $A_{13} = a_{21}a_{32} - a_{31}a_{22}.$ Here $A_{ij}$ denotes the adjoint of the element $a_{ij}$ of $A.$ Put $\xi_{o_{H}} = \sum_{j=1}^{3}\lambda_{j}\xi_{j},$ where $\sum_{j=1}^{3}\lambda_{j}^{2} = 1.$ Then the $G$-invariant one-form $\eta := {\bf g}(\xi,\cdot)$ is given at the origin by $\eta=\langle \kappa_{1}X + \kappa_{2} \mu_{\varepsilon} + \kappa_{3}\nu_{\varepsilon},\cdot\rangle,$ where 
\begin{equation}\label{kappai}
\kappa_{1} = \rho\lambda_{1},\quad \kappa_{2} = \frac{1}{a_{22}}(\lambda_{2}-\rho a_{21}\lambda_{1}),\quad \kappa_{3} = \frac{1}{a_{22}a_{33}}(\rho A_{13}\lambda_{1} - a_{32}\lambda_{2} + a_{22}\lambda_{3}).
\end{equation}
Because $d\eta(u,v) = -\frac{1}{2}\eta([u^{\tau_{H}},v^{\tau_{H}}]) = -\frac{1}{2}\eta([u,v]_{\overline{\mathfrak m}}),$ for all $u,v\in \overline{\mathfrak m},$ where $u^{\tau_{H}}$ and $v^{\tau_{H}}$ denote the local vector fields on $G/H$ around $o_{H},$ defined in (\ref{tautau}), with $u^{\tau_{H}}_{o_{H}} = u$ and $v^{\tau_{H}}_{o_{H}} = v,$ it follows, using (\ref{eq.ms4.3}) and (\ref{brackTCP1}), that
$$
\begin{array}{lcl}
d\eta(X,\mu_{\varepsilon}) = \frac{\kappa_{3}}{2},\quad d\eta(X,\nu_{\varepsilon}) = -\frac{\kappa_{2}}{2}, & &
d\eta(\mu_{\varepsilon},\nu_{\varepsilon}) = \frac{\kappa_{1}}{2},\\[0.4pc]
d\eta(\mu^{j,0}_{\varepsilon/2},\mu^{j,1}_{\varepsilon/2}) = -d\eta(\nu^{j,0}_{\varepsilon/2},\nu^{j,1}_{\varepsilon/2}) = \frac{\kappa_{3}}{4}, & & d\eta(\mu^{j,0}_{\varepsilon/2},\nu^{j,0}_{\varepsilon/2}) = d\eta(\mu^{j,1}_{\varepsilon/2},\nu^{j,1}_{\varepsilon/2}) = \frac{\kappa_{1}}{4},\\[0.4pc]
d\eta(\mu^{j,0}_{\varepsilon/2},\nu^{j,1}_{\varepsilon/2}) = -d\eta(\mu^{j,1}_{\varepsilon/2},\nu^{j,0}_{\varepsilon/2}) = -\frac{\kappa_{2}}{4}, & &
\end{array}
$$
the rest of components of $d\eta$ at $o_{H}$ being zero. From here, the $G$-invariant $(1,1)$-tensor field $\varphi$ such that $d\eta = {\bf g}(\cdot,\varphi \cdot)$ is given by
\begin{equation}\label{varphi}
\varphi \xi_{i} = \frac{1}{2}\sum_{l=1}^{3}\kappa_{l}\sum_{j,k = 1}^{3}\epsilon_{i,j,k}A_{jl}\xi_{k},\quad i\in \{1,2,3\}.
\end{equation}
 Putting $\varkappa_{1} = a_{33}\kappa_{2}- a_{32} \kappa_{3}$ and $\varkappa_{2} = a_{33}(a_{22}\kappa_{1}-a_{21}\kappa_{2}) + \kappa_{3}A_{13},$ we obtain
$$
\begin{array}{l}
\varphi\xi_{1}  =  \frac{1}{2\rho}(-a_{22}\kappa_{3}\xi_{2} + \varkappa_{1}\xi_{3}),\quad \varphi\xi_{2}  =  \frac{1}{2\rho}({\kappa_{3}}a_{22}\xi_{1} -\rho\varkappa_{2}\xi_{3}),\quad\varphi\xi_{3}  =  \frac{1}{2\rho}(-\varkappa_{1}\xi_{1} + \rho\varkappa_{2} \xi_{2}),\\[0.4pc]
\varphi \mu^{j,0}_{\varepsilon/2}  = -\frac{\kappa_{3}}{4a_{\varepsilon/2}}\mu^{j,1}_{\varepsilon/2}-\frac{\kappa_{1}}{4b_{\varepsilon/2}}\nu^{j,0}_{\varepsilon/2} +  \frac{\kappa_{2}}{4b_{\varepsilon/2}}\nu^{j,1}_{\varepsilon/2},\quad \varphi \mu^{j,1}_{\varepsilon/2} = \frac{\kappa_{3}}{4a_{\varepsilon/2}}\mu^{j,0}_{\varepsilon/2} - \frac{\kappa_{2}}{4b_{\varepsilon/2}}\nu^{j,0}_{\varepsilon/2} - \frac{\kappa_{1}}{4b_{\varepsilon/2}}\nu^{j,1}_{\varepsilon/2},\\[0.4pc]

\varphi \nu^{j,0}_{\varepsilon/2}  =  \frac{\kappa_{1}}{4a_{\varepsilon/2}}\mu^{j,0}_{\varepsilon/2} + \frac{\kappa_{2}}{4a_{\varepsilon/2}}\mu^{j,1}_{\varepsilon/2} + \frac{\kappa_{3}}{4b_{\varepsilon/2}}\nu^{j,1}_{\varepsilon/2},\quad\varphi \nu^{j,1}_{\varepsilon/2}  = -\frac{\kappa_{2}}{4a_{\varepsilon/2}}\mu^{j,0}_{\varepsilon/2} + \frac{\kappa_{1}}{4a_{\varepsilon/2}}\mu^{j,1}_{\varepsilon/2}-\frac{\kappa_{3}}{4b_{\varepsilon/2}}\nu^{j,0}_{\varepsilon/2}.
\end{array}
$$
Hence, using a straightforward calculation, $\varphi$ satisfies $\varphi^{2} = -I + \eta\otimes \xi,$ or equivalently ${\bf g}$ is compatible with the structure $(\varphi,\xi,\eta),$ and so $(\xi,{\bf g})$ is a contact metric structure, if and only if the following conditions hold:
\begin{equation}\label{e2}
\left\{
\begin{array}{l}
a_{22}^{2}\kappa_{3}^{2} + \varkappa_{1}^{2} = 4\rho^{2}(1-\lambda_{1}^{2}),\quad a_{22}^{2}\kappa_{3}^{2} + \rho^{2}\varkappa_{2}^{2} = 4\rho^{2}(1-\lambda_{2}^{2}),\\[0.4pc]
\varkappa_{1}^{2} + \rho^{2}\varkappa_{2}^{2} = 4\rho^{2}(1-\lambda_{3}^{2});\\[0.4pc]
\varkappa_{1}\varkappa_{2} = 4\rho\lambda_{1}\lambda_{2},\quad a_{22}\kappa_{3}\varkappa_{2} = 4\rho\lambda_{1}\lambda_{3}, \quad a_{22}\kappa_{3}\varkappa_{1} = 4\rho^{2}\lambda_{2}\lambda_{3};\\[0.4pc]
a_{\varepsilon/2}(\kappa_{1}^{2} + \kappa^{2}_{2}) + b_{\varepsilon/2}\kappa^{2}_{3} = 16a_{\varepsilon/2}^{2}b_{\varepsilon/2},\quad
b_{\varepsilon/2}(\kappa_{1}^{2} + \kappa_{2}^{2}) + a_{\varepsilon/2} \kappa^{2}_{3}= 16a_{\varepsilon/2}b_{\varepsilon/2}^{2},\\[0.4pc]
 \kappa_{1}\kappa_{3}(a_{\varepsilon/2}-b_{\varepsilon/2}) = \kappa_{2}\kappa_{3}(a_{\varepsilon/2}-b_{\varepsilon/2}) = 0.
 \end{array}
 \right.
 \end{equation}

Moreover, because $\varphi \xi = 0,$ we also have
\begin{equation}\label{esupletoria}
a_{22}\lambda_{2}\kappa_{3} = \lambda_{3}\varkappa_{1},\quad \rho\lambda_{3}\varkappa_{2} = a_{22}\lambda_{1}\kappa_{3},\quad \lambda_{1}\varkappa_{1} = \rho\lambda_{2}\varkappa_{2}.
\end{equation}

\vspace{0.1cm}

\noindent{\bf Type AI:} Using (\ref{kappai}) and (\ref{esupletoria}), it follows that $\kappa_{3} = 0$ and $\varkappa_{1} = 0.$ Hence, $\kappa_{2} = 0.$ Then we have that $a_{21} = A_{13} = 0$ and also $a_{31} = 0.$ Now, (\ref{e2}) hold if and only if
\[
a_{22}a_{33} = \frac{2}{\rho}, \quad \rho^{2} = 16a_{\varepsilon/2}b_{\varepsilon/2}.
\]
Then, putting $\kappa = \rho,$ $q_{\varepsilon} = \frac{a_{22}^{2}\kappa}{2},$ $\alpha =-\frac{a_{32}\kappa}{2a_{22}}$ and $q_{\varepsilon/2} = \frac{\kappa}{4a_{\varepsilon/2}},$ and using (\ref{e1}) the theorem is proved for this case.

\vspace{0.1cm}

 \noindent {\bf Type AII:} From (\ref{kappai}) and (\ref{e2}), it follows that $\kappa_{1} = \kappa_{3} = 0$ and then $a_{32} = a_{21} =  \varkappa_{2} = 0.$ Hence, the equations in (\ref{e2}) reduce to 
 \[
 a_{33} = 2\rho a_{22},\quad \frac{1}{a_{22}^{2}}= 16a_{\varepsilon/2}b_{\varepsilon/2}.
 \]
 Then the result for this case is obtained from (\ref{e1}), putting $\kappa = \frac{1}{a_{22}},$ $q_{\varepsilon} = \frac{1}{2a_{22}\rho^{2}},$ $\alpha = \frac{\rho}{2a_{22}^{2}}A_{13}$ and $q_{\varepsilon/2} = \frac{1}{4a_{22}a_{\varepsilon/2}}.$

\vspace{0.1cm}

\noindent {\bf Type AIII:} From (\ref{kappai}), we get $\kappa_{1} = \kappa_{2}= 0$ and $\kappa_{3} = \pm\frac{1}{a_{33}},$ and from (\ref{esupletoria}), it follows that $\varkappa_{1} = \varkappa_{2} = 0.$ Then, $A_{13} = a_{32} = 0$ and also $a_{31} = 0.$ Hence, (\ref{e2}) holds if and only if 
\[
a_{22} = 2\rho a_{33},\quad a_{\varepsilon/2} = b_{\varepsilon/2} = \frac{1}{4a_{33}}.
\]
Now, putting $\kappa = \frac{1}{a_{33}},$ $q_{\varepsilon} = \frac{1}{2\rho^{2}a_{33}}$ and $\alpha = -\frac{\rho^{2}a_{21}}{a_{22}}$ and using (\ref{e1}) we get the theorem for this case.

\vspace{0.1cm}

\noindent{\bf Type BI:} From (\ref{kappai}) and (\ref{esupletoria}) it follows that $\kappa_{3} = 0$ and $\lambda_{1}\varkappa_{1} = \rho\lambda_{2}\varkappa_{2}.$ Then (\ref{e2}) holds if moreover $\varkappa_{2}^{2} = 4\lambda_{1}^{2}$ and $\kappa_{1}^{2} + \kappa_{2}^{2} = 16a_{\varepsilon/2}b_{\varepsilon/2}.$ Then, putting $\kappa = \frac{1}{a_{22}},$ $\alpha = A_{13}$ and $q_{\varepsilon}  = \frac{1}{\lambda_{2}}(\lambda_{2}-\rho a_{21}\lambda_{1}),$ we have
\begin{equation}\label{BI1}
a_{31} = -\kappa\alpha q_{\varepsilon}, \quad a_{21} = \frac{\lambda_{2}(1-q_{\varepsilon})}{\rho\lambda_{1}}.
\end{equation}
On the other hand, the condition $\lambda_{1}\varkappa_{1} = \rho\lambda_{2}\varkappa_{2}$ is equivalent to $\rho^{2} \lambda_{1}^{2}= \kappa^{2}q_{\varepsilon}(1-q_{\varepsilon}\lambda_{2}^{2}).$ Hence, using (\ref{kappai}), $q_{\varepsilon} = \frac{\kappa_{1}^{2} + \kappa_{2}^{2}}{\kappa^{2}}$ and consequently, $a_{\varepsilon/2}b_{\varepsilon/2} = \frac{\kappa^{2}q_{\varepsilon}}{16}.$ Moreover, we have that $q_{\varepsilon}<\frac{1}{\lambda^{2}}.$ We put
\begin{equation}\label{BI2}
a_{\varepsilon/2} = \frac{\kappa\sqrt{q_{\varepsilon}}}{4q_{\varepsilon/2}},\quad b_{\varepsilon/2} = \frac{\kappa\sqrt{q_{\varepsilon}}q_{\varepsilon/2}}{4},
\end{equation}
where $q_{\varepsilon/2}$ is a positive constant. Finally, the condition $\varkappa_{2}^{2} = 4\lambda_{1}^{2}$ is equivalent to $(\rho\lambda_{1} - a_{21}\kappa^{2}\lambda_{2}q_{\varepsilon})^{2}a_{33}^{2} = 4\kappa^{2}\lambda_{1}^{2}.$ Then, putting $\lambda_{1} = \cos\theta$ and $\lambda_{2} = \sin\theta$ and $\varrho = \sqrt{q_{\varepsilon}(1-q_{\varepsilon}\sin^{2}\theta)},$ we have
\begin{equation}\label{BI3}
\rho = \frac{\kappa\varrho}{|\cos\theta |}, \quad a_{33} = \frac{2\varrho}{q_{\varepsilon}|\cos\theta |}.
\end{equation}
Now, from (\ref{e1}), using (\ref{BI1}), (\ref{BI2}) and (\ref{BI3}), the result is proved for this case.

\vspace{0.1cm}


\noindent{\bf Type BII:} Put $\kappa = \frac{1}{a_{33}}$ and $\beta = \frac{1}{\lambda_{3}}(\rho A_{13}\lambda_{1} + a_{22}\lambda_{3}).$ Then, from (\ref{kappai}), we get
\begin{equation}\label{k1}
\kappa_{3} = \frac{\kappa \beta}{a_{22}}\lambda_{3},\quad A_{13} = \frac{(\beta -a_{22})\lambda_{3}}{\rho\lambda_{1}},
\end{equation}
and, using (\ref{esupletoria}), we have $\varkappa_{1} = 0$ and $\rho\lambda_{3}\varkappa_{2} = a_{22}\lambda_{1}\kappa_{3}.$ Here, the condition $\varkappa_{1} = 0$ is equivalent to ${\rho\lambda_{1}}a_{21} = -a_{32}\kappa^{2}\beta\lambda_{3}.$ Then (\ref{e2}) reduces to
\begin{equation}\label{e3}
\left\{
\begin{array}{l}
\rho\lambda_{3}\varkappa_{2} = a_{22}\lambda_{1}\kappa_{3},\\[0.4pc]
a_{22}\kappa_{3}\varkappa_{2} = 4\rho\lambda_{1}\lambda_{3},\\[0.4pc]
a_{\varepsilon/2}^{2} =  b_{\varepsilon/2}^{2}= \frac{\kappa^{2}_{1} + \kappa_{2}^{2} + \kappa^{2}_{3}}{16}.
\end{array}
\right.
\end{equation}
From (\ref{k1}), the first two equations in (\ref{e3}) can be expressed as $\rho\varkappa_{2} = \kappa \beta\lambda_{1}$ and $\kappa \beta \varkappa_{2} = 4\rho\lambda_{1}.$ Then $\beta= \frac{2\rho}{\kappa}>0$ and $\varkappa_{2} = 2\lambda_{1}.$ Moreover, putting $q_{\varepsilon} = a_{22}$ and $\alpha = a_{32},$ we get
\[
A_{13} =  \frac{(2\rho - \kappa q_{\varepsilon})\lambda_{3}}{\kappa \rho \lambda_{1}},\quad a_{21} = -\frac{2\alpha\kappa\lambda_{3}}{\lambda_{1}},\quad \varkappa_{2} = \frac{(\rho(q_{\varepsilon}^{2}\lambda_{1} + 4(\alpha^{2}\kappa^{2} + 1)\lambda_{3}^{2}) -2\kappa q_{\varepsilon}\lambda_{3}^{2})}{\kappa q_{\varepsilon} \lambda_{1}}.
\]
Now, putting $\lambda_{1}= \cos\theta,$ $\lambda_{2} = \sin\theta$ and $\varrho = q_{\varepsilon}^{2}\cos^{2}\theta + 4(\alpha^{2}\kappa^{2} + 1)\sin^{2}\theta,$ we have that $\rho  = \frac{2\kappa q_{\varepsilon}}{\varrho}$ and, moreover,
\[
A_{13} = \frac{4-\varrho}{2\kappa}\tan\theta, \quad a_{21}  =  -2\alpha\kappa\tan\theta,\quad a_{31} = \frac{q_{\varepsilon}^{2}-4(1 + \alpha^{2}\kappa^{2})}{2\kappa q_{\varepsilon}}\sin\theta\cos\theta.
\] 
Then the result for this case follows applying again (\ref{e1}).

\vspace{0.1cm}

\noindent{\bf Type BIII:} From (\ref{kappai}) and (\ref{esupletoria}), we get $\kappa_{1} = \varkappa_{2} = 0.$ Moreover, putting $\kappa = \frac{1}{a_{33}}$ and $\beta = \frac{1}{\lambda_{3}}(a_{22}\lambda_{3} - a_{32}\lambda_{2}),$ it follows that $a_{21} = \frac{\kappa^{2}A_{13} \beta\lambda_{3}}{\lambda_{2}}.$ On the other hand, the equations in (\ref{e2}) reduces to
\begin{equation}\label{e4}
\left\{
\begin{array}{l}
a_{22}^{2}\kappa_{3}^{2} + \varkappa_{1}^{2} = 4\rho^{2},\\[0.4pc]
a_{22}^{2}\kappa_{3}^{2} = 4\rho^{2}\lambda_{3}^{2},\\[0.4pc]
a_{22}\kappa_{3}\varkappa_{1} = 4\rho^{2}\lambda_{2}\lambda_{3},\\[0.4pc]
a_{\varepsilon/2}^{2} = b_{\varepsilon/2}^{2} = \frac{\kappa_{2}^{2} + \kappa_{3}^{2}}{16}.
\end{array}
\right.
\end{equation}
Taking into account that $\kappa_{3}=\frac{\kappa\beta}{a_{22}}\lambda_{3}$ and $\varkappa_{1} = \frac{1}{\kappa a_{22}}(\lambda_{2} - a_{32}\kappa^{2}\beta\lambda_{3}),$ we get from the second equation in (\ref{e4}) that $\kappa^{2}\beta^{2} = 4\rho^{2}.$ Then, from the third equation we have  $\varkappa_{1} = \kappa\beta\lambda_{2}$ and the first equation holds. Hence, we obtain that $a_{22} = \frac{1}{\kappa^{2}\beta}(\lambda_{2}^{2} + 4\rho^{2}\lambda_{3}^{2})$ and it implies that $\beta>0,$ that is, $\beta = \frac{2\rho}{\kappa}.$ Now, putting $\lambda_{2} = \cos\theta,$ $\lambda_{3} = \sin\theta,$ $q_{\varepsilon} = 2\rho$ and $\alpha = \frac{2\kappa A_{13}\rho}{\varrho},$ where $\varrho = \cos^{2}\theta + 4\rho^{2}\sin^{2}\theta,$ we have
$$
\begin{array}{l}
a_{21} = \frac{2\alpha\varrho}{q_{\varepsilon}}\tan\theta,\quad a_{22} = \frac{\varrho}{\kappa q_{\varepsilon}},\quad a_{31} = -\frac{2\alpha\varrho}{q_{\varepsilon}},\\[0.4pc]
 a_{32} =\frac{1-q_{\varepsilon}^{2}}{\kappa q_{\varepsilon}}\sin\theta\cos\theta, \quad a_{\varepsilon/2} = b_{\varepsilon/2} = \frac{\kappa q_{\varepsilon}\sqrt{\varrho}}{4\varrho}.
\end{array}
$$
Then, from (\ref{e1}), we get the result for this case.

\vspace{0.1cm}

\noindent{\bf Type C:} Using (\ref{esupletoria}), we get
\begin{equation}\label{varkappas}
\varkappa_{1} = \frac{a_{22}\kappa_{3}\lambda_{2}}{\lambda_{3}},\quad \varkappa_{2} = \frac{a_{22}\kappa_{3}\lambda_{1}}{\rho\lambda_{3}}.
\end{equation}
Then, from the first equation in (\ref{e2}), we have $a_{22}^{2}\kappa^{2}_{3} = 4\rho^{2}\lambda_{3}^{2}.$ Hence, one gets
\[
\varkappa_{1}^{2} = 4\rho^{2}\lambda_{2}^{2},\quad \varkappa_{2}^{2} = 4\lambda_{1}^{2}
\]
and the equations in (\ref{e2}) hold if, moreover, $a_{\varepsilon/2}^{2} = b^{2}_{\varepsilon/2} = \frac{\kappa^{2}_{1} + \kappa^{2}_{2} + \kappa^{2}_{3}}{16}.$
Put $\kappa = \frac{1}{a_{33}},$ $\alpha = a_{21}$ and $\beta = \frac{1}{\lambda_{3}}(\rho A_{13}\lambda_{1} - a_{32}\lambda_{2} + a_{22}\lambda_{3}).$ Then $\kappa_{3} = \frac{\kappa\beta\lambda_{3}}{a_{22}}$ and the equations in (\ref{varkappas}) are written as $\varkappa_{1} = \kappa\beta\lambda_{2}$ and $\rho\varkappa_{2} = \kappa\beta\lambda_{1}.$ Hence, we have
\[
a_{32} = \frac{(1-\kappa^{2}\beta a_{22})\lambda_{2}-\rho\alpha\lambda_{1}}{\kappa^{2}\beta\lambda_{3}},\quad A_{13} = \frac{\kappa^{2}\beta a_{22}\lambda_{1} + \rho(\alpha(\lambda_{2}-\rho\alpha\lambda_{1}) - a_{22}^{2}\rho\lambda_{1})}{\kappa^{2}\beta\rho\lambda_{3}}.
\]
 Moreover, putting $q_{\varepsilon} = \rho,$ we obtain $\kappa^{2}\beta^{2} = 4 q_{\varepsilon}^{2}.$ Since $a_{22}\lambda_{3} = \beta \lambda_{3}+a_{32}\lambda_{2} - q_{\varepsilon} A_{13}\lambda_{1},$ it follows that $a_{22}$ is a solution of the second order equation $q_{\varepsilon}^{2}\lambda_{1}^{2} a_{22}^{2} - \kappa^{2}\beta a_{22} + \varrho = 0,$ where $\varrho = 4q_{\varepsilon}^{2}\lambda_{3}^{2} + \lambda_{2}^{2} + \alpha q_{\varepsilon}\lambda_{1}(\alpha q_{\varepsilon}\lambda_{1} - 2\lambda_{2}).$ Hence, it must happen that $\varrho\lambda_{1}^{2}\leq \kappa^{2}$ and $\beta>0.$ Therefore, putting $\lambda_{1} = \cos\theta\cos\phi,$ $\lambda_{2} = \sin\theta\cos\phi$ and $\lambda_{3} = \sin\phi,$ we conclude that, for each pair of positive constants $\kappa$ and $q_{\varepsilon}$ and $\alpha\in {\mathbb R}$ such that $\varrho\cos^{2}\theta\cos^{2}\phi\leq \kappa^{2},$ there exists a contact metric structure $(\xi,{\bf g})$ of Type C and the result follows, putting $\gamma_{1} = a_{32}$ and $\gamma_{2} = A_{31},$ using (\ref{e1}).
 
 \vspace{0.2cm}
 
\noindent{\bf Proof of Theorem \ref{maingeneral}.} From Theorem \ref{maingeneralCP1}, $(\xi,{\bf g})$ is orthogonal, for each Type AI-AIII, BI-BIII and C, if and only if the following conditions hold:
\begin{enumerate}
\item[$\bullet$] Type A (AI-III): $\alpha = 0,$
\item[$\bullet$] Type BI: $q_{\varepsilon} = 1$ and $\alpha = 0.$ Then, $\varrho = |\cos\theta|,$
\item[$\bullet$] Type BII: $q_{\varepsilon} = 2$ and $\alpha = 0.$ Then, $\varrho = 4,$
\item[$\bullet$] Type BIII: $q_{\varepsilon} = 1$ and $\alpha = 0.$ Then, $\varrho = 1,$
\item[$\bullet$] Type C: $\gamma_{1} = \gamma_{2} = 0$ and $\alpha = 0.$ Then, $\kappa = q_{\varepsilon} = \frac{1}{2},$ $\delta = 2$ and $\varrho = 1-\cos^{2}\theta\cos^{2}\phi.$
\end{enumerate}
From here, we obtain the coefficients $a_{0},$ $a_{\varepsilon},$ $b_{\varepsilon},$ $a_{\varepsilon/2}$ and $b_{\varepsilon/2}$ of ${\bf g}$, given in Table II. 

Next, let ${\mathfrak U}\colon \overline{\mathfrak m}\times \overline{\mathfrak m}\to \overline{\mathfrak m}$ be the symmetric bilinear function defined in (\ref{U}) for the Euclidean space $(\overline{\mathfrak m},{\bf g}_{o_{H}}),$ where ${\bf g}$ is an orthogonal metric on $T_{r}{\mathbb C}{\bf P}^{n}.$ Using  (\ref{eq.ms4.3}), (\ref{Bb}) and (\ref{brackTCP1}), together with Lemma \ref{pbrack}, we have
 \begin{equation}\label{UCP}
\begin{array}{l}
{\mathfrak U}(X,\mu_{\varepsilon}) = \frac{a_{0}-a_{\varepsilon}}{2b_{\varepsilon}}\nu_{\varepsilon},\quad {\mathfrak U}(X,\nu_{\varepsilon}) = \frac{b_{\varepsilon}-a_{0}}{2a_{\varepsilon}}\mu_{\varepsilon},\quad
{\mathfrak U}(X,\mu^{j,a}_{\varepsilon/2}) = \frac{a_{0}-a_{\varepsilon/2}}{4b_{\varepsilon/2}}\nu^{j,a}_{\varepsilon/2},\\[0.4pc]

{\mathfrak U}(X,\nu^{j,a}_{\varepsilon/2}) = \frac{b_{\varepsilon/2}-a_{0}}{4a_{\varepsilon/2}}\mu^{j,a}_{\varepsilon/2}, \quad {\mathfrak U}(\mu_{\varepsilon},\nu_{\varepsilon}) = \frac{a_{\varepsilon}-b_{\varepsilon}}{2a_{0}}X,\quad {\mathfrak U}(\mu_{\varepsilon},\mu^{j,a}_{\varepsilon/2}) = \frac{(-1)^{a}(a_{\varepsilon/2}-a_{\varepsilon})}{4b_{\varepsilon/2}}\nu^{j,a+1}_{\varepsilon/2},\\[0.4pc]

{\mathfrak U}(\mu_{\varepsilon},\nu^{j,a}_{\varepsilon/2}) = \frac{(-1)^{a}(b_{\varepsilon/2}-a_{\varepsilon})}{4a_{\varepsilon/2}}\mu^{j,a+1}_{\varepsilon/2},\quad {\mathfrak U}(\nu_{\varepsilon}, \mu^{j,a}_{\varepsilon/2}) = \frac{(-1)^{a}(b_{\varepsilon}-a_{\varepsilon/2})}{4a_{\varepsilon/2}}\mu^{j,a+1}_{\varepsilon/2},\\[0.4pc]

{\mathfrak U}(\nu_{\varepsilon},\nu^{j,a}_{\varepsilon/2}) = \frac{(-1)^{a}(b_{\varepsilon/2}-b_{\varepsilon})}{4b_{\varepsilon/2}}\nu^{j,a+1}_{\varepsilon/2},\quad {\mathfrak U}(\mu^{j,a}_{\varepsilon/2},\nu^{j,a}_{\varepsilon/2}) = \frac{(a_{\varepsilon/2}-b_{\varepsilon/2})(a_{0} + a_{\varepsilon})}{4a_{0}a_{\varepsilon}}X,\\[0.4pc]

{\mathfrak U}(\mu^{j,0}_{\varepsilon/2},\nu^{j,1}_{\varepsilon/2}) = -{\mathfrak U}(\mu^{j,1}_{\varepsilon/2},\nu^{j,0}_{\varepsilon/2}) = -\frac{(a_{\varepsilon/2}-b_{\varepsilon/2})}{4a_{\varepsilon}}\mu_{\varepsilon},
\end{array}
\end{equation}
for all $j = 1,\dots, n-1;$ $a = 0,1$ and  the rest of its components being zero. Therefore, applying (\ref{Killing}), we obtain in Table II the set of all $G$-invariant orthogonal metrics ${\bf g}$ for which $(\xi,{\bf g})$ is $K$-contact. 

\vspace{0.2cm}

Finally we prove that the contact metric structures of Type {\rm AI} and Type {\rm AII} are isomorphic for each triple of positive constants $(\kappa, q_{\varepsilon},q_{\varepsilon/2}).$ 

Denote by ${\bf g}^{\kappa,q_{\varepsilon},q_{\varepsilon/2}}_{1}$ and ${\bf g}^{\kappa,q_{\varepsilon},q_{\varepsilon/2}}_{2}$ the $G$-invariant contact metrics on $T_{r}{\mathbb C}{\bf P}^{n} = G/H$ associated to the structures of Type AI y AII, respectivamente, and consider the linear isometry $L\colon (\overline{\mathfrak m},{\bf g}^{\kappa,q_{\varepsilon},q_{\varepsilon/2}}_{1o_{H}}) \to (\overline{\mathfrak m},{\bf g}^{\kappa,q_{\varepsilon},q_{\varepsilon/2}}_{2o_{H}}),$ given by
 \begin{equation}\label{L}
 \begin{array}{l}
 LX = \mu_{\varepsilon},\quad L\mu_{\varepsilon} = -X,\quad L\nu_{\varepsilon} = \nu_{\varepsilon},\quad
 L\mu^{j,a}_{\varepsilon/2} = (-1)^{a}\mu^{j,a+1}_{\varepsilon/2},\quad L\nu^{j,a}_{\varepsilon/2} = \nu^{j,a}_{\varepsilon/2},
 \end{array}
 \end{equation}
 for all $j = 1,\dots, n-1$ and $a = 0,1.$ Let us check that $L$ satisfies the conditions (i), (ii) and (iii) in Lemma \ref{lfirst}. By a straightforward calculation, using (\ref{brackTCP1}), $L$ verifies (i). For (ii), note that the non-zero ${\mathfrak h}$-components of the brackets in (\ref{brackTCP1}) are of the form $[\mu^{j,a}_{\varepsilon/2},\mu^{k,b}_{\varepsilon/2}]_{\mathfrak h},$ for $j,k = 1,\dots, n-1$ and $a,b = 0,1,$ and we get
 \[
 [L\mu^{j,a}_{\varepsilon/2},L\mu^{k,b}_{\varepsilon/2}]_{\mathfrak h} = [\mu^{j,a}_{\varepsilon/2},\mu^{k,b}_{\varepsilon/2}]_{\mathfrak h}.
 \]
 Then, from Lemma \ref{lh}, the condition (ii) is equivalent to
 \begin{equation}\label{eii}
 L\circ {\rm ad}_{u} = {\rm ad}_{u}\circ L,\quad u\in {\mathfrak h},
 \end{equation}
 which is satisfied using Lemma \ref{hbrackets}. From (\ref{L}), (iii) is verified. Finally, because ${\rm SU}(n+1)$ is simply connected and ${\rm S}({\rm U}(1)\times{\rm U}(n-1))$ is connected, diffeomorphic to ${\rm SU}(1)\times{\rm SU}(n-1)\times {\mathbb S}^{1},$ it follows that $T_{r}({\mathbb C}{\bf P}^{n})$ is also simply connected.  
 
  \section{Invariant Sasakian and 3-Sasakian structures}
 
 \setcounter{equation}{0} 
 
From Theorem \ref{maingeneral}, there are three families of orthogonal $K$-contact structures $(\xi^{\kappa}_{i},{\bf g}^{\kappa}_{i}),$ $i=1,2,3,$ on $T_{r}{\mathbb C}{\bf P}^{n},$ depending of a parameter $\kappa>0,$ and determined at $o_{H}$ by
$$
\begin{array}{lclclcl}
\xi^{\kappa}_{1}  =  \pm\frac{1}{\kappa}X, & & {\bf g}^{\kappa}_{1} & = & \kappa^{2}\langle\cdot,\cdot\rangle_{\mathfrak a} + \frac{\kappa}{2}\langle\cdot,\cdot\rangle_{{\mathfrak m}_{\varepsilon}\oplus{\mathfrak k}_{\varepsilon}} + \frac{\kappa}{4}\langle\cdot,\cdot\rangle_{{\mathfrak m}_{\varepsilon/2}\oplus{\mathfrak k}_{\varepsilon/2}},\\[0.4pc]

\xi^{\kappa}_{2}  =  \pm\frac{1}{\kappa}\mu_{\varepsilon}, & & {\bf g}^{\kappa}_{2} & = & \kappa^{2}\langle\cdot,\cdot\rangle_{{\mathfrak m}_{\varepsilon}} + \frac{\kappa}{2}\langle\cdot,\cdot\rangle_{{\mathfrak a}\oplus{\mathfrak k}_{\varepsilon}} + \frac{\kappa}{4}\langle\cdot,\cdot\rangle_{{\mathfrak m}_{\varepsilon/2}\oplus{\mathfrak k}_{\varepsilon/2}},\\[0.4pc]

\xi^{\kappa}_{3} = \pm\frac{1}{\kappa}\nu_{\varepsilon}, & &{\bf g}^{\kappa}_{3} & = & \kappa^{2}\langle\cdot,\cdot\rangle_{{\mathfrak k}_{\varepsilon}} + \frac{\kappa}{2}\langle\cdot,\cdot\rangle_{{\mathfrak a}\oplus {\mathfrak m}_{\varepsilon}} + \frac{\kappa}{4}\langle\cdot,\cdot\rangle_{{\mathfrak m}_{\varepsilon/2}\oplus{\mathfrak k}_{\varepsilon/2}}.
\end{array}
$$
They are, respectively, of Type AI, AII and AIII. Since $(\xi^{\kappa}_{1},{\bf g}^{\kappa}_{1})$ is Sasakian \cite[Theorem 1.1]{JC}, it follows from the last part of Theorem \ref{maingeneral} that $(\xi^{\kappa}_{2},{\bf g}^{\kappa}_{2})$ is also Sasakian. However, as can be seen in the following proposition, the procedure used in the proof of Theorem \ref{maingeneral} is not applicable when the contact metric structure $(\xi^{\kappa}_{3},{\bf g}^{\kappa}_{3})$ is considered.
 \begin{proposition} There is no isomorphism $L$ between the canonical infinitesimal models $(\overline{\mathfrak m},\widetilde{T}_{o_{H}},\widetilde{R}_{o_{H}},{\bf g}^{\kappa}_{1\,o_{H}})$ and $(\overline{\mathfrak m},\widetilde{T}_{o_{H}},\widetilde{R}_{o_{H}},{\bf g}^{\kappa}_{3\,o_{H}})$ of $T_{r}{\mathbb C}{\bf P}^{n} = G/H,$ for $n\geq 2,$ such that $L{\mathfrak a} = {\mathfrak k}_{\varepsilon}.$
\end{proposition}
\begin{proof} From the condition (\ref{eii}) and taking into account that $H= S(U(1)\times U(n-1))$ is connected, it follows that any isomorphism $L$ of infinitesimal models on $T_{r}{\mathbb C}{\bf P}^{n}$ is ${\rm Ad}(H)$-invariant. Hence, from the irreducibility of the decomposition $\overline{\mathfrak m} ={\mathfrak a}\oplus{\mathfrak m}_{\varepsilon}\oplus{\mathfrak k}_{\varepsilon} \oplus{\mathfrak m}_{\varepsilon/2}\oplus{\mathfrak k}_{\varepsilon/2},$ either $L {\mathfrak m}_{\varepsilon/2} = {\mathfrak m}_{\varepsilon/2}$ or $L{\mathfrak m}_{\varepsilon/2} = {\mathfrak k}_{\varepsilon/2}.$ Proceeding by reductio ad absurdum, suppose that there exists an isomorphism $L$ such that $L{\mathfrak a} = {\mathfrak k}_{\varepsilon}.$ Then, if $L{\mathfrak m}_{\varepsilon/2} = {\mathfrak k}_{\varepsilon/2}$ and so $L{\mathfrak k}_{\varepsilon/2} = {\mathfrak m}_{\varepsilon/2},$ we have from Lemma \ref{lbrack1} that
\[
L[{\mathfrak m}_{\varepsilon/2},{\mathfrak m}_{\varepsilon/2}]_{\overline{\mathfrak m}}\subset L{\mathfrak k}_{\varepsilon},\quad [L{\mathfrak m}_{\varepsilon/2}, L{\mathfrak m}_{\varepsilon/2}]\subset {\mathfrak k}_{\varepsilon}.
\]
Hence, using (i) in Lemma \ref{lfirst} and that $[{\mathfrak m}_{\varepsilon/2},{\mathfrak m}_{\varepsilon/2}]_{\overline{\mathfrak m}}\neq 0$ by (\ref{brackTCP1}), $L{\mathfrak k}_{\varepsilon} = {\mathfrak k}_{\varepsilon}.$ But it is absurd since by hypothesis $L{\mathfrak a} = {\mathfrak k}_{\varepsilon}.$ 

If $L{\mathfrak m}_{\varepsilon/2} = {\mathfrak m}_{\varepsilon/2}$ and so $L{\mathfrak k}_{\varepsilon/2} = {\mathfrak k}_{\varepsilon/2},$ then using again Lemma \ref{lbrack1},
\[
L[{\mathfrak a},{\mathfrak m}_{\varepsilon/2}]_{\overline{\mathfrak m}}\subset L{\mathfrak k}_{\varepsilon/2} = {\mathfrak k}_{\varepsilon/2},\quad [L{\mathfrak a},L{\mathfrak m}_{\varepsilon/2}]_{\overline{\mathfrak m}} = [{\mathfrak k}_{\varepsilon},{\mathfrak m}_{\varepsilon/2}]_{\overline{\mathfrak m}}\subset {\mathfrak m}_{\varepsilon/2}.
\]
Therefore, it must happen that $[{\mathfrak a},{\mathfrak m}_{\varepsilon/2}]_{\overline{\mathfrak m}}= [{\mathfrak k}_{\varepsilon},{\mathfrak m}_{\varepsilon/2}]_{\overline{\mathfrak m}}=0,$ which is absurd because $n\geq 2.$

\end{proof}

\begin{theorem}\label{t1} The $K$-contact structures $(\xi^{\kappa}_{i},{\bf g}^{\kappa}_{i})$ on $T_{r}{\mathbb C}{\bf P}^{n},$ for $i = 1,2,3$ and $\kappa >0,$ are Sasakian. Moreover, the tangent sphere bundle $T_{r}G/K$ of a compact rank-one symmetric space $G/K$ admits an $G$-invariant $3$-Sasakian structure if and only if $G/K= {\mathbb C}{\bf P}^{n},$ for some $n\geq 1.$ Then there exists a unique orthogonal $3$-Sasakian metric ${\bf g}$ on $T_{r}{\mathbb C}{\bf P}^{n},$ given at $o_{H}$ by
 \begin{equation}\label{ff}
{\bf g}_{o_{H}} = \frac{1}{4}\langle\cdot,\cdot\rangle_{\overline{\mathfrak n}} + \frac{1}{8}\langle\cdot,\cdot\rangle_{{\mathfrak m}_{\varepsilon/2}\oplus{\mathfrak k}_{\varepsilon/2}}.
 \end{equation}
  \end{theorem}

\begin{remark}{\rm From Remark \ref{r1}, note that $T_{r}{\mathbb S}^{2}$ and $T_{r}{\mathbb R}{\bf P}^{2},$ as particular cases of Theorem \ref{t1}, admit an invariant $3$-Sasakian structure.}
\end{remark}

\vspace{0.1cm}

\noindent{\bf Proof of Theorem \ref{t1}.} First, by a direct procedure, we will prove that $(\xi_{3}^{\kappa} = \pm\frac{1}{\kappa}\nu_{\varepsilon},{\bf g}^{\kappa}_{3})$ is normal. From (\ref{varphi}), the $(1,1)$-tensor field $\varphi_{3}$ of the corresponding almost contact metric structure is determined by
 \begin{equation}\label{3}
\begin{array}{l}
\varphi_{3}X = -\mu_{\varepsilon},\quad \varphi_{3}\mu_{\varepsilon} = X,\quad \varphi_{3}\nu_{\varepsilon} = 0,\\[0.4pc]
\varphi_{3}\mu^{j,a}_{\varepsilon/2} = (-1)^{a+1}\mu^{j,a+1}_{\varepsilon/2},\quad \varphi_{3}\nu^{j,a}_{\varepsilon/2} = (-1)^{a}\nu^{j,a+1}_{\varepsilon/2},
\end{array}
\end{equation}
for all $j = 1,\dots,n-1;$ $a = 0,1,$ and, using (\ref{brackTCP1}), also as
\[
\varphi_{{3}_{\mid {\mathfrak m}_{\lambda} \oplus{\mathfrak k}_{\lambda}} }= \frac{1}{\lambda_{\mathbb R}(X)}{\rm ad}_{\nu_{\varepsilon}},\quad \varphi_{3}X = {\rm ad}_{\nu_{\varepsilon}}X,
\]
where $\lambda\in \Sigma^{+} = \{\varepsilon,\varepsilon/2\}.$ Then, since the subspaces $\overline{\mathfrak n}$ and ${\mathfrak m}_{\varepsilon/2}\oplus {\mathfrak k}_{\varepsilon/2}$ of $\overline{\mathfrak m}$ are ${\rm ad}_{\nu_{\varepsilon}}$-invariant, it follows that 
\begin{equation}\label{comm}
{\rm ad}_{\nu_{\varepsilon}}\circ\varphi_{3} = \varphi_{3}\circ {\rm ad}_{\nu_{\varepsilon}}.
\end{equation}
 On the other hand, taking into account that $2d\eta^{\kappa}_{3}(u,v) = -\eta^{\kappa}_{3}([u,v]_{\overline{\mathfrak m}}),$ for all $u,v\in \overline{\mathfrak m},$ where $\eta^{\kappa}_{3}= {\bf g}^{\kappa}_{3}(\xi^{\kappa}_{3},\cdot),$ we get that the (skew-symmetric) $G$-invariant tensor field ${\bf N}_{3}= [\varphi_{3},\varphi_{3}] + 2d\eta^{\kappa}_{3}\otimes\xi^{\kappa}_{3}$ is determined by
\begin{equation}\label{N}
{\bf N}_{3}(u,v) = -[u,v]_{\overline{\mathfrak m}} + [\varphi_{3}u,\varphi_{3}v]_{\overline{\mathfrak m}}- \varphi_{3}[\varphi_{3}u,v]_{\overline{\mathfrak m}} - \varphi_{3}[u,\varphi_{3}v]_{\overline{\mathfrak m}}.
\end{equation}
From (\ref{comm}) and using that the orthogonal complement $(\xi_{3}^{\kappa})^{\bot}$ in $\overline{\mathfrak m}$ of $\xi_{3}^{\kappa}$ is ${\rm ad}_{\nu_{\varepsilon}}$-preserving, it follows  that ${\bf N}_{3}(\xi^{\kappa}_{3},\cdot) = 0.$ Moreover, from (\ref{N}),
\[
{\bf N}_{3}(\varphi_{3}u,v) = {\bf N}_{3}(u,\varphi_{3}v), \quad u,v\in (\xi^{\kappa}_{3})^{\bot},
\]
and so, ${\bf N}_{3}(u,\varphi_{3}u) = 0.$ Then, applying (\ref{3}), we have ${\bf N}_{3}(X,\mu_{\varepsilon}) = {\bf N}_{3}(\mu^{j,a}_{\varepsilon/2},\mu^{j,a+1}_{\varepsilon/2}) = {\bf N}_{2}(\nu^{j,a}_{\varepsilon/2},\nu^{j,a+1}_{\varepsilon/2}) = 0$ and
\begin{equation}\label{NN}
\begin{array}{lcl}
{\bf N}_{3}(X,\mu^{j,a}_{\varepsilon/2}) & = & (-1)^{a+1}{\bf N}_{3}(\mu_{\varepsilon},\mu^{j,a+1}_{\varepsilon/2}),\\[0.4pc]
{\bf N}_{3}(X,\nu^{j,a}_{\varepsilon/2}) &  = & (-1)^{a}{\bf N}_{3}(\mu_{\varepsilon},\nu^{j,a+1}_{\varepsilon/2}),\\[0.4pc]
{\bf N}_{3}(\mu^{j,a}_{\varepsilon/2},\mu^{k,b}_{\varepsilon/2}) & = & (-1)^{a+b+1}{\bf N}_{3}(\mu^{j,a+1}_{\varepsilon/2},\mu^{k,b+1}_{\varepsilon/2}),\\[0.4pc]
{\bf N}_{3}(\nu^{j,a}_{\varepsilon/2},\nu^{k,b}_{\varepsilon/2})&  = & (-1)^{a+b+1}{\bf N}_{3}(\nu^{j,a+1}_{\varepsilon/2},\nu^{k,b+1}_{\varepsilon/2}),\\[0.4pc]
 {\bf N}_{3}(\mu^{j,a}_{\varepsilon/2},\nu^{k,b}_{\varepsilon/2}) & = & (-1)^{a+b}{\bf N}_{3}(\mu^{j,a+1}_{\varepsilon/2},\nu^{k,b+1}_{\varepsilon/2}),
\end{array}
\end{equation}
for all $j,k=1,\dots, n-1$ and $a,b= 0,1.$ Using  (\ref{brackTCP1}), we have for $j\neq k,$
\[
[\mu^{j,a}_{\varepsilon/2},\nu^{k,b}_{\varepsilon}] = [\mu^{j,a}_{\varepsilon/2},\mu^{k,b}_{\varepsilon/2}]_{\overline{\mathfrak m}} = [\nu^{j,a}_{\varepsilon/2},\nu^{k,b}_{\varepsilon/2}]_{\overline{\mathfrak m}} = 0.
\]
Hence, from (\ref{3}), (\ref{N}) and (\ref{NN}), we get
\[
{\bf N}_{3}(\mu^{j,a}_{\varepsilon/2},\mu^{k,b}_{\varepsilon/2}) = {\bf N}_{3}(\nu^{j,a}_{\varepsilon/2},\nu^{k,b}_{\varepsilon/2})=  0,
\]
for all $a,b=0,1,$ and $j,k= 1,\dots, n-1.$ Moreover, for $j\neq k,$ we obtain ${\bf N}_{3}(\mu^{j,a}_{\varepsilon/2},\nu^{k,b}_{\varepsilon/2}) = 0.$ Then we conclude that ${\bf N}_{3}$ is identically null if and only if
\[
{\bf N}_{3}(X,\mu^{j,a}_{\varepsilon/2}) = {\bf N}_{3}(X,\nu^{j,a}_{\varepsilon/2}) = {\bf N}_{3}(\mu^{j,a}_{\varepsilon/2},\nu^{j,b}_{\varepsilon/2}) = 0,\quad a\leq b,
\]
 which is satisfied by direct calculation using (\ref{brackTCP1}) and (\ref{3}) in (\ref{N}). This proves that $(\xi_{3}^{\kappa} = \pm\frac{1}{\kappa},{\bf g}^{\kappa}_{3})$ is also Sasakian. 

For $\kappa = \frac{1}{2},$ the three Sasakian metrics ${\bf g}^{\kappa}_{i},$ $i = 1,2,3,$ coincide with the metric ${\bf g}$ determined in (\ref{ff}) and clearly $(\{\xi_{1},\xi_{2},\xi_{3}\},{\bf g})$ is a $G$-invariant $3$-Sasakian structure on $T_{r}{\mathbb C}{\bf P}^{n},$ where $\xi_{1\,{o_{H}}}=2X,$ $\xi_{2\,{o_{H}}}=2\nu_{\varepsilon},$ and $\xi_{3\,{o_{H}}}= 2\mu_{\varepsilon}.$ From Theorem \ref{maingeneral}, ${\bf g}$ is the unique $G$-invariant orthogonal $3$-Sasakian metric on tangent sphere bundles of compact rank-one symmetric spaces, thus proving the result.

\begin{theorem}\label{t2}The $3$-Sasakian metric ${\bf g}$ is the unique orthogonal Sasakian-Einstein metric on $T_{r}{\mathbb C}{\bf P}^{n}.$  
 \end{theorem}
\begin{proof} From Theorem \ref{t1}, all orthogonal Sasakian metrics on $T_{r}{\mathbb C}{\bf P}^{n}$ are of the form ${\bf g}^{\kappa}_{i},$ for $i = 1,2,3$ and $\kappa>0.$ 

Let $\alpha_{i}\colon \overline{\mathfrak m}\times\overline{\mathfrak m}\to \overline{\mathfrak m}$ be the ${\rm Ad}(H)$-invariant bilinear function that determines the Levi-Civita connection of ${\bf g}^{\kappa}_{i}.$ Then, applying (\ref{brackTCP1}) and (\ref{UCP}) in (\ref{nabla}), we have that the non-vanishing components of $\alpha_{1}$ and $\alpha_{3}$ are given by
 $$
\begin{array}{l}
\alpha_{1}(X,\mu_{\varepsilon}) = (\kappa -1)\nu_{\varepsilon},\quad \alpha_{1}(\mu_{\varepsilon},X) = \kappa\nu_{\varepsilon},\quad \alpha_{1}(X,\nu_{\varepsilon}) = -\alpha_{3}(\nu_{\varepsilon},X) =(1-\kappa)\mu_{\varepsilon},\\[0.4pc]

 \alpha_{1}(\nu_{\varepsilon},X) = -\alpha_{3}(X,\nu_{\varepsilon}) = -\kappa\xi_{\varepsilon},\quad \alpha_{3}(X,\mu_{\varepsilon}) = -\alpha_{3}(\mu_{\varepsilon},X) = -\frac{1}{2}\nu_{\varepsilon},\\[0.4pc]

\alpha_{1}(X,\mu^{j,a}_{\varepsilon/2}) = \frac{(2\kappa -1)}{2}\mu^{j,a}_{\varepsilon/2},\quad \alpha_{1}(\mu^{j,a}_{\varepsilon/2},X) = \kappa\nu^{j,a}_{\varepsilon/2},\quad
\alpha_{1}(X,\nu^{j,a}_{\varepsilon/2}) = \frac{1-2\kappa}{2}\mu^{j,a}_{\varepsilon/2},\\[0.4pc]
 
\alpha_{1}(\nu^{j,a}_{\varepsilon/2},X) =-\kappa\mu^{j,a}_{\varepsilon/2},\quad \alpha_{1}(\mu_{\varepsilon},\nu_{\varepsilon}) = -\alpha_{1}(\nu_{\varepsilon},\mu_{\varepsilon}) = -\frac{1}{2}X,\\[0.4pc]

\alpha_{1} (\mu^{j,a}_{\varepsilon/2},\mu_{\varepsilon}) = \alpha_{1}(\nu^{j,a}_{\varepsilon/2},\nu_{\varepsilon}) = \alpha_{3}(\mu^{j,a}_{\varepsilon/2},\mu_{\varepsilon}) = \frac{(-1)^{a+1}}{2}\nu_{\varepsilon/2}^{j,a+1},\\[0.4pc]

\alpha_{1}(\mu^{j,a}_{\varepsilon/2},\nu_{\varepsilon}) = -\alpha_{1}(\nu^{j,a}_{\varepsilon/2},\mu_{\varepsilon}) = -\alpha_{3}(\nu^{j,a}_{\varepsilon/2},\mu_{\varepsilon}) = \frac{(-1)^{a}}{2}\mu^{j,a+1}_{\varepsilon/2},\\[0.4pc]

 \alpha_{3}(\mu^{j,a}_{\varepsilon/2},X) = \frac{1}{2}\nu^{j,a}_{\varepsilon/2},\quad \alpha_{3}(\nu^{j,a}_{\varepsilon/2},X) =-\frac{1}{2}\mu^{j,a}_{\varepsilon/2},\;\;\alpha_{3}(\mu_{\varepsilon},\nu_{\varepsilon}) = -\kappa X,\quad \alpha_{3}(\nu_{\varepsilon},\mu_{\varepsilon}) = (1-\kappa)X,\\[0.4pc]

 \alpha_{3}(\nu^{j,a}_{\varepsilon/2},\nu_{\varepsilon}) = (-1)^{a+1}\kappa\nu_{\varepsilon/2}^{j,a+1},\quad \alpha_{3}(\mu^{j,a}_{\varepsilon/2},\nu_{\varepsilon}) = (-1)^{a}\kappa\mu^{j,a+1}_{\varepsilon/2},\\[0.4pc]

\alpha_{3}(\nu_{\varepsilon},\mu^{j,a}_{\varepsilon/2}) = (-1)^{a}(\kappa -1)\mu^{j,a+1}_{\varepsilon/2},\quad \alpha_{3}(\nu_{\varepsilon},\nu^{j,a}_{\varepsilon/2}) = (-1)^{a}\frac{1-2\kappa}{2}\nu^{j,a+1}_{\varepsilon/2},
\end{array}
$$
for all $j= 1,\dots, n-1,$ together, for $i=1,2,3.$ with 
$$
\begin{array}{l}
\alpha_{i}(\mu^{j,0}_{\varepsilon/2},\mu^{j,1}_{\varepsilon/2}) = -\alpha_{i}(\nu^{j,0}_{\varepsilon/2},\nu^{j,1}_{\varepsilon/2}) = -\frac{1}{4}\nu_{\varepsilon},\\[0.4pc]

\alpha_{i}(\mu^{j,0}_{\varepsilon/2},\nu^{j,1}_{\varepsilon/2}) = -\alpha_{i}(\mu^{j,1}_{\varepsilon/2},\nu^{j,0}_{\varepsilon/2}) = -\alpha_{i}(\nu^{j,0}_{\varepsilon/2},\mu^{j,1}_{\varepsilon/2}) = \alpha_{i}(\nu^{j,1}_{\varepsilon/2},\mu^{j,0}_{\varepsilon/2}) = \frac{1}{4}\mu_{\varepsilon}.
\end{array}
$$

Denote by $R_{i}$ the Riemannian curvature, expressed as $(0,4)$-tensor field, and the Ricci curvature ${\rm Ric}_{i}$ of $(T_{r}{\mathbb C}{\bf P}^{n},{\bf g}^{\kappa}_{i}).$ Then, using (\ref{R}) and (\ref{brackTCP1}), we get 
$$
\begin{array}{lclcl}
R_{1}(\mu_{\varepsilon},X,\mu_{\varepsilon},X) & = & R_{3}(X,\nu_{\varepsilon},X,\nu_{\varepsilon})  & = & \frac{\kappa^{3}}{2},\\[0.4pc]
R_{1}(\mu_{\varepsilon},\nu_{\varepsilon},\mu_{\varepsilon},\nu_{\varepsilon}) & =  & R_{3}(X,\mu_{\varepsilon},X,\mu_{\varepsilon}) & =  &\frac{\kappa(2-3\kappa)}{4},\\[0.4pc]
R_{1}(\mu_{\varepsilon},\mu^{j,a}_{\varepsilon/2},\mu_{\varepsilon},\mu^{j,a}_{\varepsilon/2})  & =  & R_{1}(\mu_{\varepsilon},\nu^{j,a}_{\varepsilon/2},\mu_{\varepsilon},\nu^{j,a}_{\varepsilon/2}) &  = &\frac{\kappa}{16},\\[0.4pc]
 R_{3}(X,\mu^{j,a}_{\varepsilon/2},X,\mu^{j,a}_{\varepsilon/2}) & =  &R_{3}(X,\nu^{j,a}_{\varepsilon/2},X,\nu^{j,a}_{\varepsilon/2})  & = &\frac{\kappa}{16},
\end{array}
$$
for all $j = 1,\dots,n-1,$ $a = 0,1.$ Hence, ${\rm Ric}_{1}(\mu_{\varepsilon},\mu_{\varepsilon}) = {\rm Ric}_{3}(X,X) = n-\kappa.$ On the other hand, because $(\xi_{1}^{\kappa},{\bf g}_{1}^{\kappa})$ and $(\xi^{\kappa}_{3},{\bf g}^{\kappa}_{3})$ are Sasakian structures, the condition to be ${\bf g}_{1}^{\kappa}$ or ${\bf g}_{3}^{\kappa}$ Einstein implies that the Einstein constant $\lambda$ must be equals to $2(2n-1)$ and so ${\rm Ric}_{1}(\mu_{\varepsilon}, \mu_{\varepsilon}) = {\rm Ric}_{3}(X,X)= (2n-1)\kappa.$ Hence, $\kappa = \frac{1}{2}.$ We get the same result for the families of metrics ${\bf g}^{\kappa}_{2}$ because, according with Theorem \ref{maingeneral}, the Sasakian structures $(\xi^{\kappa}_{1},{\bf g}^{\kappa}_{1})$ and $(\xi^{\kappa}_{2},{\bf g}^{\kappa}_{2})$ are isomorphic for every $\kappa.$
\end{proof}

\section{Invariant Einstein metrics on $T_{r}{\mathbb S}^{n}$ and $T_{r}{\mathbb R}{\mathbf P}^{n}$ }

\setcounter{equation}{0}

Consider the unit sphere ${\mathbb S}^{n}$ of ${\mathbb R}^{n+1}$ as the quotient $G/K = {\rm SO}(n+1)/{\rm SO}(n)$ and the real projective space ${\mathbb R}{\bf P}^{n}$ expressed as $G/K = {\rm SO}(n+1)/{\rm S}({\rm O}(1)\times {\rm O}(n)).$ According with \cite[Proposition 4.2]{JC}, $T_{r}G/K$ is the quotient manifold $G/H,$ where $H = {\rm SO}(n-1)$ for ${\mathbb S}^{n}$ and $H = {\rm S}({\rm O}(1)\times{\rm O}(n-1))$ for ${\mathbb R}{\bf P}^{n}.$

The set of matrices $\{B^{0}_{jk},\,1\leq j<k\leq n+1\}$ is an orthonormal basis of the Lie algebra ${\mathfrak s}{\mathfrak o}(n+1)$ of $G = SO(n+1),$ with respect to the invariant trace form $\langle A,B\rangle = -\frac{1}{2}{\rm trace}\, AB,$ for all $A,B\in {\mathfrak s}{\mathfrak o}(n+1).$ Then the Lie algebra ${\mathfrak k}={\mathfrak s}{\mathfrak o}(n)$ of $K,$ where $K = {\rm SO}(n)$ or $K = {\rm S}({\rm O}(1)\times {\rm O}(n)),$ and its orthogonal complement ${\mathfrak m}$ in ${\mathfrak s}{\mathfrak o}(n+1)$ are given by
\[
{\mathfrak k} = {\mathbb R}\{B^{0}_{jk}\}_{2\leq j<k\leq n+1},\quad {\mathfrak m}= {\mathbb R}\{B^{0}_{1k}\}_{2\leq k\leq n+1}.
\]
 Fixed as Cartan subspace to ${\mathfrak a} = {\mathbb R}X,$ where $X= B^{0}_{12},$ the centralizer ${\mathfrak h}$ of ${\mathfrak a}$ in ${\mathfrak k}$ is, using (\ref{bracABC}), trivial for $n = 2$ and, for $n\geq 3,$ it is given by 
\begin{equation}\label{hh}
{\mathfrak h} = {\mathbb R}\{B^{0}_{jk}\}_{3\leq j<k\leq n+1}.
\end{equation}
Since ${\rm ad}^{2}_{X}B^{0}_{1k} = -B^{0}_{1k},$ for $k = 3,\dots ,n+1,$ it follows that ${\mathfrak m}_{\varepsilon} ={\mathbb R}\{B^{0}_{1k}\}_{k = 3,\dots, n+1}$ and, because ${\rm ad}_{X}B^{0}_{1k} = -B^{0}_{2k},$ we put
 \begin{equation}\label{xizeta}
 \mu^{j}_{\varepsilon} = B^{0}_{1\,2+j},\quad \nu^{j}_{\varepsilon} = B^{0}_{2\,2+j},\quad j = 1,\dots, n-1.
 \end{equation}
 Then $\{\mu^{j}_{\varepsilon}\}$ and $\{\nu^{j}_{\varepsilon}\}$ are orthonormal basis of ${\mathfrak m}_{\varepsilon}$ and ${\mathfrak k}_{\varepsilon},$ respectively, satisfying (\ref{eq.ms4.3}). From (\ref{bracABC}), we have
 \begin{equation}\label{brackS}
 [\mu^{j}_{\varepsilon},\mu^{k}_{\varepsilon}] = [\nu^{j}_{\varepsilon},\nu^{k}_{\varepsilon}] = -B^{0}_{2+j\,2+k},\quad [\mu^{j}_{\varepsilon},\nu^{k}_{\varepsilon}] = -\delta_{jk}X.
 \end{equation}
 Hence the bilinear function ${\mathfrak U}$ for the Euclidean space $(\overline{\mathfrak m},{\bf g}_{o_{H}}),$ where $\overline{\mathfrak m} = \overline{\mathfrak n} = {\mathfrak a}\oplus{\mathfrak m}_{\varepsilon}\oplus{\mathfrak k}_{\varepsilon}$ and ${\bf g}_{o_{H}} = {\bf g}_{o_{H}}^{a_{0},a_{\varepsilon},b_{\varepsilon}}= a_{0}\langle\cdot,\cdot\rangle_{\mathfrak a} + a_{\varepsilon}\langle\cdot,\cdot\rangle_{{\mathfrak m}_{\varepsilon}} + b_{\varepsilon}\langle\cdot,\cdot\rangle_{{\mathfrak k}_{\varepsilon}},$ $a_{0},a_{\varepsilon}, b_{\varepsilon}\in {\mathbb R}^{+},$ is determined by
\begin{equation}\label{USR}
{\mathfrak U}(X,\mu^{j}_{\varepsilon}) = \frac{a_{0}-a_{\varepsilon}}{2b_{\varepsilon}}\nu^{j}_{\varepsilon},\quad {\mathfrak U}(X,\nu^{j}_{\varepsilon}) = \frac{b_{\varepsilon}-a_{0}}{2a_{\varepsilon}}\mu^{j}_{\varepsilon},\quad {\mathfrak U}(\mu^{j}_{\varepsilon},\nu^{k}_{\varepsilon}) = \frac{a_{\varepsilon}-b_{\varepsilon}}{2a_{0}}\delta_{jk}X,
\end{equation}
for all $j,k = 1,\dots, n-1,$ the rest of components being zero. Then, using (\ref{bracABC}) again, the function $\alpha\colon \overline{\mathfrak m}\times\overline{\mathfrak m}\to \overline{\mathfrak m}$ that determines the Levi-Civita connection $\nabla$ of $(G/H,{\bf g}),$ is given by 
$$
\begin{array}{lclclcl}
\alpha(X,\mu^{j}_{\varepsilon}) & = & \frac{(a_{0}-a_{\varepsilon}-b_{\varepsilon})}{2b_{\varepsilon}}\nu^{j}_{\varepsilon},& & \alpha(\mu^{j}_{\varepsilon},X) & = & \frac{(a_{0}-a_{\varepsilon}+b_{\varepsilon})}{2b_{\varepsilon}}\nu^{j}_{\varepsilon},\\[0.4pc]
\alpha(X,\nu^{j}_{\varepsilon}) &  = &  \frac{(a_{\varepsilon}+b_{\varepsilon}- a_{0})}{2a_{\varepsilon}}\mu^{j}_{\varepsilon}, & &
\alpha(\nu^{j}_{\varepsilon},X) & = & \frac{(b_{\varepsilon}- a_{\varepsilon} - a_{0})}{2a_{\varepsilon}}\mu^{j}_{\varepsilon},\\[0.4pc]
\alpha(\mu^{j}_{\varepsilon},\nu^{k}_{\varepsilon})&  = & \frac{(a_{\varepsilon}-b_{\varepsilon}-a_{0})}{2a_{0}}\delta_{jk}X, & & \alpha(\nu^{k}_{\varepsilon},\mu^{j}_{\varepsilon})&  = & \frac{(a_{\varepsilon}-b_{\varepsilon}+a_{0})}{2a_{0}}\delta_{jk}X,
\end{array}
$$
the rest of components being zero. Hence, applying (\ref{eq.ms4.3}), (\ref{bracABC}), (\ref{xizeta}) and (\ref{brackS}) in (\ref{R}), we have the following.
\begin{proposition}\label{pR} The Riemannian curvature tensor $R$ of $(G/H,{\bf g})$ at $o_{H}$ is given by
\begin{equation}\label{curvatures1}
\begin{array}{lcl}
R(X,\mu^{j}_{\varepsilon},X,\mu^{j}_{\varepsilon}) &  = &  \frac{1}{4b_{\varepsilon}}(2b_{\varepsilon}(a_{0} + a_{\varepsilon}-b_{\varepsilon}) + (a_{0}-a_{\varepsilon})^{2}-b^{2}_{\varepsilon});\\[0.4pc]

R(X,\nu^{j}_{\varepsilon},X,\nu^{j}_{\varepsilon}) &  =  & \frac{1}{4a_{\varepsilon}}(2a_{\varepsilon}(a_{0}-a_{\varepsilon} + b_{\varepsilon}) + (a_{0}-b_{\varepsilon})^{2} - a^{2}_{\varepsilon});\\[0.4pc]

R(\mu^{j}_{\varepsilon},\mu^{k}_{\varepsilon},\mu^{j}_{\varepsilon},\mu^{k}_{\varepsilon}) &  = &  a_{\varepsilon},\quad R(\nu^{j}_{\varepsilon},\nu^{k}_{\varepsilon},\nu^{j}_{\varepsilon},\nu^{k}_{\varepsilon}) = b_{\varepsilon},\;k\neq j;\\[0.4pc]

R(\mu^{j}_{\varepsilon},\mu^{k}_{\varepsilon},\nu^{j}_{\varepsilon},\nu^{k}_{\varepsilon}) &  = &  \frac{1}{4a_{0}}(4a_{0}b_{\varepsilon} - (a_{0} - a_{\varepsilon} + b_{\varepsilon})^{2}),\;k\neq j;\\[0.4pc]

R(\mu^{j}_{\varepsilon},\nu^{j}_{\varepsilon},\mu^{k}_{\varepsilon},\nu^{l}_{\varepsilon}) &  = &  -\frac{1}{4a_{0}}(2a_{0}(a_{0}-a_{\varepsilon}-b_{\varepsilon})\delta_{kl} + (a_{0}^{2}-(a_{\varepsilon}-b_{\varepsilon})^{2})\delta_{jk}\delta_{jl});\\[0.4pc]

R(\mu^{j}_{\varepsilon},\nu^{k}_{\varepsilon},\mu^{k}_{\varepsilon},\nu^{j}_{\varepsilon}) &  =&  -\frac{1}{4a_{0}}(a_{0}^{2}-(a_{\varepsilon}-b_{\varepsilon})^{2}),\; k\neq j,
\end{array}
\end{equation}
 the rest of its components being zero.
\end{proposition}

\begin{corollary}\label{c1} The $\xi^{S}$-sectional curvature on $(T_{r}G/K= G/H,{\bf g})$ is a {\rm (}positive{\rm )} constant $c$ if and only if either $a_{\varepsilon} = b_{\varepsilon},$ then $c = \frac{a_{0}}{4a_{\varepsilon}^{2}},$ or $a_{0} = a_{\varepsilon} + b_{\varepsilon},$ then $c = \frac{1}{a_{\varepsilon} + b_{\varepsilon}}.$
\end{corollary}
\begin{proof} From (\ref{curvatures1}) it follows that the $\xi^{S}$-sectional curvatures $K(X,\mu^{j}_{\varepsilon})$ and  $K(X,\nu_{\varepsilon}^{j})$ at the origin $o_{H}$ coincide, for some $j = 1,\dots, n-1,$ if and only if $a_{\varepsilon} = b_{\varepsilon}$ or $a_{0} = a_{\varepsilon} + b_{\varepsilon}.$ Put $c = K(X,\mu^{j}_{\varepsilon}) = K(X,\nu^{j}_{\varepsilon}).$ Then $c = \frac{a_{0}}{4a_{\varepsilon}^{2}}$ if $a_{\varepsilon} = b_{\varepsilon}$ and $c =  \frac{1}{a_{\varepsilon} + b_{\varepsilon}}$ if $a_{0} = a_{\varepsilon} + b_{\varepsilon}.$ 

Next, let $u$ be any vector in $\overline{\mathfrak m}$ orthogonal to $X.$ Then, $u = \sum_{j=1}^{n-1}(u_{j}\mu^{j}_{\varepsilon} + \overline{u}_{j}\nu^{j}_{\varepsilon}),$ for some $u_{j},\overline{u}_{j}\in {\mathbb R}.$ From Proposition \ref{pR}, we get
$$
\begin{array}{lcl}
R(X,u,X,u) & = & \sum_{j=1}^{n-1}((u_{j})^{2}R(X,\mu^{j}_{\varepsilon},X,\mu^{j}_{\varepsilon}) +(\overline{u}_{j})^{2}R(X,\nu^{j}_{\varepsilon},X,\nu^{j}_{\varepsilon}))\\[0.4pc]
& =  & c\,a_{0}(a_{\varepsilon}\sum_{j=1}^{n-1}(u_{j})^{2} + b_{\varepsilon}\sum_{j=1}^{n-1}(\overline{u}_{j})^{2}) = c\,a_{0}{\bf g}(u,u).
\end{array}
$$
Hence, $K(X,u) = c.$ Because ${\bf g}$ is a homogeneous Riemannian metric, this implies that the $\xi^{S}$-sectional curvature on $T_{r}G/K$ is constant equals to $c.$
\end{proof}

Next, let $Q$ be the Ricci operator of $(G/H,{\bf g}),$ defined by ${\rm Ric}(\cdot,\cdot) = g(Q\cdot,\cdot).$ Then one gets that the vectors $\{X,\mu^{j}_{\varepsilon},\nu^{j}_{\varepsilon};\;j = 1,\dots,n-1\}$ form an orthogonal basis of eigenvectors of $Q$ at $o_{H}.$  Moreover, we have
\begin{lemma}\ The subspaces ${\mathfrak a},$ ${\mathfrak m}_{\varepsilon}$ and ${\mathfrak k}_{\varepsilon}$ of $\overline{\mathfrak m}$ are eigenspaces of the Ricci operator $Q$ at $o_{H}$ and the corresponding eigenvalues $\rho_{0},$ $\rho_{\varepsilon}$ and $\varrho_{\varepsilon}$ are given by 
\begin{equation}\label{eigenvalues}
\begin{array}{lcl}
\rho_{0} & = & \frac{n-1}{2a_{0}a_{\varepsilon}b_{\varepsilon}}(a_{0}^{2}-(a_{\varepsilon}-b_{\varepsilon})^{2}),\\[0.4pc]
\rho_{\varepsilon} & = & \frac{1}{2a_{0}a_{\varepsilon} b_{\varepsilon}}(2(n-1)a_{0}b_{\varepsilon} + a^{2}_{\varepsilon}-b_{\varepsilon}^{2}-a_{0}^{2}),\\[0.4pc]

\varrho_{\varepsilon} & = & \frac{1}{2a_{0}a_{\varepsilon}b_{\varepsilon}}(2(n-1)a_{0}a_{\varepsilon} - a^{2}_{\varepsilon} + b_{\varepsilon}^{2}-a_{0}^{2}).
\end{array}
\end{equation}

\end{lemma}
\begin{remark}{\rm From this lemma, the scalar curvature $s$ of $(G/H,{\bf g})$ is given by}
\[
s = \frac{n-1}{2a_{0}a_{\varepsilon}b_{\varepsilon}}(2(n-1)a_{0}(a_{\varepsilon}+ b_{\varepsilon}) - (a_{\varepsilon} - b_{\varepsilon})^{2} - a_{0}^{2}).
\]
\end{remark}

Now we are in a position to determine all the invariant metrics that are Einstein on $T_{r}{\mathbb S}^{n}$ and $T_{r}{\mathbb R}{\bf P}^{n}.$

\begin{theorem}\label{teinstein1} There exists, except for a homothetic coefficient $\alpha>0,$ a unique $G$-invariant Einstein Riemannian metric ${\bf g}^{\alpha}$ on $T_{r}(G/K),$ for $G/K = {\mathbb S}^{n}$ or ${\mathbb R}{\bf P}^{n}$ $(n\geq 2),$ and it is determined  by
\begin{equation}\label{einstein}
{\bf g}^{\alpha}_{o_{H}} = \alpha\Big(\langle\cdot,\cdot\rangle_{\mathfrak a} +  \frac{n}{2(n-1)}\langle\cdot,\cdot\rangle_{{\mathfrak m}_{\varepsilon}\oplus{\mathfrak k}_{\varepsilon}}\Big).
\end{equation}
\end{theorem} 

\begin{proof} Let ${\bf g} = {\bf g}^{a_{0},a_{\varepsilon},b_{\varepsilon}}$ be any $G$-invariant Riemannian metric on $T_{r}(G/K) = G/H.$ From (\ref{eigenvalues}), it follows that $\rho_{\varepsilon} = \varrho_{\varepsilon}$ if and only if 
\begin{equation}\label{conditions}
a_{\varepsilon} = b_{\varepsilon}\quad \mbox{\rm or}\quad a_{0} = \frac{1}{n-1}(a_{\varepsilon} + b_{\varepsilon}).
\end{equation}
If $a_{\varepsilon} = b_{\varepsilon},$ clearly the equality $\rho_{0} = \rho_{\varepsilon} = \varrho_{\varepsilon}$ holds if and only if $a_{\varepsilon} =\frac{n}{2(n-1)}a_{0}$ and, putting $\alpha= a_{0},$ we have that ${\bf g} = {\bf g}^{\alpha}.$ If $a_{0}= \frac{1}{n-1}(a_{\varepsilon} + b_{\varepsilon}),$ then 
$$
\begin{array}{lcl}
\rho_{0} & = & \frac{1}{2(a_{\varepsilon} + b_{\varepsilon})a_{\varepsilon}b_{\varepsilon}}((a_{\varepsilon} + b_{\varepsilon})^{2} - (n-1)^{2}(a_{\varepsilon}-b_{\varepsilon})^{2}),\\[0.4pc]
\rho_{\varepsilon} & = & \frac{n(n-2)}{2(n-1)a_{\varepsilon}b_{\varepsilon}}(a_{\varepsilon} + b_{\varepsilon}).
\end{array}
$$
Let us see that $\rho_{0}= \rho_{\varepsilon}$ cannot occur in this case. Here,
\[
\Big(\frac{n(n-2)}{n-1} - 1\Big )(a_{\varepsilon} + b_{\varepsilon})^{2} + (n-1)^{2}(a_{\varepsilon}-b_{\varepsilon})^{2} = 0.
\]
Then, for $n = 2$ we get that $a_{\varepsilon}b_{\varepsilon} = 0$ and, for $n>2,$ that $a_{\varepsilon} = b_{\varepsilon} = 0,$ but $a_{\varepsilon}$ and $b_{\varepsilon}$ are positive constants.
\end{proof}

From Corollary \ref{c1}, Theorem \ref{teinstein1} and (\ref{USR}), we have directly the following.
\begin{corollary}The standard vector field $\xi^{S}$ is Killing with constant $\xi^{S}$-sectional curvature equals to $\frac{(n-1)^{2}}{n^{2}\alpha}$ on the tangent sphere bundle $T_{r}(G/K),$ for $G/K = {\mathbb S}^{n}$ or ${\mathbb R}{\bf P}^{n},$ equipped with the $G$-invariant Einstein metric ${\bf g}^{\alpha}.$
\end{corollary}

\begin{corollary}\label{c2} There exists an invariant Riemannian metric on $T_{r}{\mathbb S}^{n}$ or $T_{r}{\mathbb R}{\bf P}^{n}$ with cons\-tant sectional curvature if and only if $n =2.$ On $T_{r}{\mathbb S}^{2} \cong SO(3)$ or on $T_{r}{\mathbb R}{\bf P}^{2}\cong SO(3)/\pm I_{2},$ an invariant Riemannian metric ${\bf g}$ has constant sectional curvature if and only if ${\bf g}_{o_{H}} = \alpha\langle\cdot,\cdot\rangle,$ for some $\alpha>0.$ Then, the constant sectional curvature $c$ is $\frac{1}{4\alpha}.$
\end{corollary}
\begin{proof} Suppose that $(G/H,{\bf g})$ has constant sectional curvature $c.$ From Corollary \ref{c1}, $a_{\varepsilon} = b_{\varepsilon}$ or $a_{0} = a_{\varepsilon} + b_{\varepsilon}.$ If $a_{\varepsilon} = b_{\varepsilon},$ then $c = \frac{a_{0}}{4a_{\varepsilon}^{2}}$ and from (\ref{curvatures1}),
\[
c = K(\mu^{1}_{\varepsilon},\nu^{1}_{\varepsilon}) = \frac{4a_{\varepsilon}-3a_{0}}{4a^{2}_{\varepsilon}}.
\]
Hence, $a_{0} = a_{\varepsilon} = b_{\varepsilon}$ and $c = \frac{1}{4\alpha},$ where $\alpha = a_{0}.$ If $a_{0} = a_{\varepsilon} + b_{\varepsilon},$ then $K(\mu^{1}_{\varepsilon},\nu^{1}_{\varepsilon}) = -\frac{1}{2a_{\varepsilon}}<0$ and this case does not occur. If $n\geq 3,$ we have that $K(\mu^{1}_{\varepsilon},\mu^{2}_{\varepsilon}) = \frac{1}{a_{\varepsilon}} = \frac{1}{\alpha}\neq \frac{1}{4\alpha}.$ 

Then $n = 2$ and ${\bf g}_{o_{H}} = \alpha\langle\cdot,\cdot\rangle,$ for some $\alpha>0.$ Moreover, from Theorem \ref{teinstein1}, ${\bf g}$ is Einstein and hence, ${\bf g}$ has in fact constant sectional curvature.
\end{proof}

\begin{corollary}The induced Sasaki metric $\tilde{g}^{S}$ on $T_{r}(G/K),$ for $G/K = {\mathbb S}^{n}$ or ${\mathbb R}{\bf P}^{n},$ is Einstein if and only if $r =1$ and $n = 2.$ Then $(T_{1}{\mathbb S}^{2},\tilde{g}^{S})$ and $(T_{1}{\mathbb R}{\bf P}^{2},\tilde{g}^{S})$ have constant sectional curvature $c = \frac{1}{4}.$
\end{corollary}
\begin{proof} From (\ref{gsasaki}), $\tilde{g}^{S} = {\bf g}^{\alpha},$ for some $\alpha>0,$ if and only if $r = 1$ and $n = 2.$ This implies that $\alpha = 1$ and $\tilde{g}^{S}_{o_{H}} = \langle\cdot,\cdot\rangle.$ From Corollary \ref{c2}, $\tilde{g}^{S}$ has constant sectional curvature $c = \frac{1}{4}.$
\end{proof}

 \begin{theorem}\label{teinsteinS1} There exists a unique $G$-invariant contact metric ${\bf g}$ on $T_{r}(G/K),$ for $G/K = {\mathbb S}^{n}$ or ${\mathbb R}{\bf P}^{n},$ $n\geq 2,$ which is Einstein. Such a metric is Sasakian and it is determined by
 \[
 {\bf g}_{o_{H}} = \frac{(n-1)^{2}}{n^{2}}\Big( \langle\cdot,\cdot\rangle_{\mathfrak a} + \frac{n}{2(n-1)}\langle\cdot,\cdot\rangle_{{\mathfrak m}_{\varepsilon}\oplus{\mathfrak k}_{\varepsilon}}\Big).
 \]
 If $n = 2,$ ${\bf g}$ is the $3$-Sasakian metric ${\bf g}_{o_{H}} = \frac{1}{4}\langle\cdot,\cdot\rangle,$ with constant sectional curvature $c = 1.$
\end{theorem}

\begin{proof} From Theorem \ref{maingeneral} case (i), if $G/K\in \{{\mathbb S}^{n},{\mathbb R}{\bf P}^{n}\;(n\geq 3)\},$ any $G$-invariant contact metric structure $(\xi,{\bf g})$  on $T_{r}(G/K)$ is determined by $\xi_{o_{H}} = \pm \frac{1}{\kappa}X$ and ${\bf g}_{o_{H}} = \kappa^{2}\langle\cdot,\cdot\rangle_{\mathfrak a} + \frac{\kappa}{2q_{\varepsilon}}\langle\cdot,\cdot\rangle_{{\mathfrak m}_{\varepsilon}} + \frac{\kappa q_{\varepsilon}}{2}\langle\cdot,\cdot\rangle_{{\mathfrak k}_{\varepsilon}},$ for some $\kappa, q_{\varepsilon}>0.$ Then, using Theorem \ref{teinstein1}, ${\bf g}$ is Einstein if and only if $q_{\varepsilon} = 1,$ or equivalently $(\xi,{\bf g})$ is Sasakian, and $\kappa = \frac{n-1}{n}.$ For $n =2,$ we apply Theorem \ref{maingeneral} case (ii) and again Theorem \ref{teinstein1}.
\end{proof}


\begin{thebibliography}{99}
 \setlength{\baselineskip}{0.4cm} 
 
 \bibitem{BaHs} Back, A. and Hsiang, W.Y., Equivariant geometry and Kervaire spheres, {\em Trans. Amer. Math. Soc.} {\bf 304} (1) (1987), 207-227.

\bibitem{Be1} Besse, A.L., \emph{Manifolds all of whose Geodesics are Closed}, Springer, Berlin, Heidelberg, 1978.  

\bibitem{Bl} Blair, D.E., {\em Riemannian geometry of contact
and symplectic manifolds}, Progress in Math. {\bf 203},
Birk-h\"auser, 2002.

\bibitem{BoyerGalicki:3Sasakian} Boyer, C.P. and Galicki, K., 3-Sasakian manifolds, Surveys in
Differential Geometry, VI: Essays on Einstein manifolds, {\em Int. Press}, Boston, MA (1999), 123-184.

\bibitem{BoyerGalicki:Sasakian-Einstein} Boyer, C.P. and Galicki, K., On Sasakian-Einstein Geometry, \emph{Int. J. Math.} {\bf 11} (7) (2000), 873-909.

\bibitem{BoyerGalicki} Boyer, C.P. and Galicki, K., Einstein manifolds and contact geometry, {\em Proc. Amer. Math. Soc.}  {\bf 129} (8) (2001), 2419-2430.

\bibitem{BoyerGalickilibro} Boyer, C.P. and Galicki, K., {\em Sasakian Geometry}, Oxford Mathematical Monographs, Oxford University Press, Oxford, 2008.

\bibitem{BGM} Boyer, C.P., Galicki, K. and Mann, B.M., The geometry and topology of 3-Sasakian manifolds, {\em J. Reine Angew. Math.} {\bf 455}  (1994), 183-220.

\bibitem{ChG} Chinea, D. and Gonz\'alez-D\'avila, J.C., A classification of almost contact metric manifolds, {\em Ann. Mat. Pura Appl.} {\bf 156} (4) (1990), 15-34 .

\bibitem{GGV} Gil-Medrano, O., Gonz\'alez-D\'avila, J.C. and Vanhecke, L., Harmonic and minimal invariant unit vector fields on homogeneous Riemannian manifolds. {\em Houston J. Math.} {\bf 27} (2) (2001), 377-409.

\bibitem{GRS} Goertsches, O., Roschig, L. and Stecker, L., Revisiting the classification of homogeneous 3-Sasakian and quaternionic K\"ahler manifolds, {\em European Journal of Mathematics}, {\bf 9} (11) (2023).

\bibitem{JC} Gonz\'alez-D\'avila, J.C., Sasakian structures on tangent sphere bundles of compact rank-one symmetric spaces, {\em Mediterr. J. Math.} {\bf 19} (227) (2022).

\bibitem{JC1} Gonz\'alez-D\'avila, J.C., Invariant contact metric structures on tangent sphere bundles of compact symmetric spaces, {\em Rev. Real Acad. Cienc. Exactas Fis. Nat. Ser. A-Mat.} {\bf 117} (142) (2023).

\bibitem{He} Helgason, S., {\em Differential geometry, Lie groups, and symmetric spaces}, Academic Press,  New York, San Francisco, London
(Pure and Applied Mathematics, a Series of Monographs and Textbooks),
1978.

\bibitem{Ke} Kerr, M., New examples of homogeneous Einstein metrics, {\em Michigan J. Math.}  {\bf 45} (1) (1998), 115-134.

\bibitem{KNI} Kobayashi, S. and Nomizu, K., {\em Foundations of differential geometry}, I, Interscience publishers, New York, London, 1963.

\bibitem{Kow} Kowalski, O., {\em Generalized symmetric spaces}, Lecture Notes in Math. {\bf 805}, Springer-Verlag, Berlin, Heidelberg, New York, 1980.

\bibitem{Loos} Loos, O., {\em Symmetric Spaces II: Compact Spaces and Classification}, W.A. Benjamin, Inc., New York, 1969.

\bibitem{N} Nomizu, K., Invariant affine connections on homogeneous spaces, {\em Amer. J. Math.} {\bf 76} (1) (1954), 33-65.

\bibitem{Sparks} Sparks, J., Sasaki-Einstein Manifolds, {\em Surveys Diff. Geom.} {\bf 16} (2011), 265-324.

\bibitem{Tash} Tashiro, Y., On contact structures of tangent sphere bundles, {\em Tohoku Math. J.} {\bf 21} (1969), 117-143.

\bibitem{TV} Tricerri, F. and Vanhecke, L., Curvature homogeneous Riemannian manifolds, {\em Ann. Scient. Éc. Norm. Sup.} {\bf 22} (1989), 535-554.

\end{thebibliography}
\end{document}